\documentclass[journal]{IEEEtran}
\usepackage{amsmath,amsfonts,amsthm,amssymb,color}
\usepackage{mathrsfs}
\usepackage{fancybox}
\usepackage{tikz}
\usepackage{enumerate}
\usepackage{bm}
\usepackage{subeqnarray}
\usepackage{cases}
\usepackage{pifont}
\usepackage{graphicx}
\usepackage{epstopdf}
\usepackage{authblk}
\usepackage[procnumbered,ruled,boxed]{algorithm2e}
\usepackage{etoc}

\DontPrintSemicolon
\SetKw{KwAnd}{and}
\SetProcNameSty{textsc}
\SetFuncSty{textsc}
\SetKwInOut{Input}{Input}
\SetKwInOut{Output}{Output}
\SetKwRepeat{Do}{do}{while}
\usepackage{graphics}
\usepackage{hyperref}
\usepackage{cleveref}
\usepackage{thm-restate}

\usepackage{bbm}
\usepackage{tablefootnote}

\usepackage[margin=1.0in]{geometry}

\newtheorem{theorem}{Theorem}
\newtheorem{Assumption}{Theorem}
\newtheorem{Definition}{Theorem}

\newtheorem{lemma}[theorem]{Lemma}

\newtheorem{assumption}[Assumption]{Assumption}

\theoremstyle{definition}
\newtheorem{definition}[Definition]{Definition}
\newtheorem{remark}[theorem]{Remark}

\newenvironment{fminipage}%
{\begin{Sbox}\begin{minipage}}%
		{\end{minipage}\end{Sbox}\fbox{\TheSbox}}

\def\seq#1{\begin{equation*}\small\begin{split}#1\end{split}\end{equation*}}
\def\seql#1#2{\begin{equation}{#1}\small\begin{split}#2\end{split}\end{equation}}

\newenvironment{eq}{\begin{equation}}{\end{equation}}

\def\prob#1{\mbox{\rm \textbf{Pr}}\left[ #1 \right]}

\def\expec#1{{\mathbb{E}}\left[ #1 \right]}

\def\defeq{\stackrel{\mathrm{def}}{=}}

\def\pr#1{\left( #1 \right ) }
\def\br#1{\left[ #1 \right ] }
\def\dr#1{\left\{#1\right\}}
\def\jr#1{\left<#1\right>}

\def\floor#1{\left\lfloor #1 \right\rfloor}

\def\abs#1{\left|#1  \right|}

\def\calA{\mathcal{A}}
\def\calB{\mathcal{B}}

\def\calD{\mathcal{D}}
\def\calE{\mathcal{E}}
\def\calF{\mathcal{F}}
\def\calG{\mathcal{G}}

\def\calT{\mathcal{T}}

\def\calN{\mathcal{N}}
\def\calR{\mathcal{R}}
\def\calV{\mathcal{V}}
\def\calP{\mathcal{P}}
\def\calW{\mathcal{W}}
\def\calZ{\mathcal{Z}}

\newcommand\PPi{\boldsymbol{\Pi}}

\newcommand\aalpha{\boldsymbol{\alpha}}

\def\aa{\pmb{\mathit{a}}}
\newcommand\bb{\boldsymbol{\mathit{b}}}

\newcommand\dd{\boldsymbol{\mathit{d}}}
\newcommand\ee{\boldsymbol{\mathit{e}}}

\renewcommand\gg{\boldsymbol{\mathit{g}}}

\newcommand\pp{\boldsymbol{\mathit{p}}}
\newcommand\qq{\boldsymbol{\mathit{q}}}

\renewcommand\ss{\boldsymbol{\mathit{s}}}

\newcommand\uu{\boldsymbol{\mathit{u}}}
\newcommand\vv{\boldsymbol{\mathit{v}}}
\newcommand\ww{\boldsymbol{\mathit{w}}}
\newcommand\yy{\boldsymbol{\mathit{y}}}
\newcommand\zz{\boldsymbol{\mathit{z}}}
\newcommand\xx{\boldsymbol{\mathit{x}}}

\renewcommand\AA{\boldsymbol{\mathit{A}}}
\newcommand\BB{\boldsymbol{\mathit{B}}}
\newcommand\CC{\boldsymbol{\mathit{C}}}
\newcommand\DD{\boldsymbol{\mathit{D}}}
\newcommand\EE{\boldsymbol{\mathit{E}}}
\newcommand\FF{\boldsymbol{\mathit{F}}}

\newcommand\II{\boldsymbol{\mathit{I}}}
\newcommand\JJ{\boldsymbol{\mathit{J}}}

\newcommand\MM{\boldsymbol{\mathit{M}}}

\newcommand\PP{\boldsymbol{\mathit{P}}}
\newcommand\QQ{\boldsymbol{\mathit{Q}}}
\newcommand\RR{\boldsymbol{\mathit{R}}}

\newcommand\UU{\boldsymbol{\mathit{U}}}
\newcommand\WW{\boldsymbol{\mathit{W}}}
\newcommand\VV{\boldsymbol{\mathit{V}}}
\newcommand\XX{\boldsymbol{\mathit{X}}}
\newcommand\YY{\boldsymbol{\mathit{Y}}}

\newcommand\xxhat{\widehat{\boldsymbol{\mathit{x}}}}
\newcommand\XXhat{\widehat{\boldsymbol{\mathit{X}}}}

\newcommand \yyhat{\widehat{\yy}}

\def\nm#1#2{\left\| #2 \right\|_{#1}}

\def\nt#1{\left\| #1 \right\|_2}
\def\ni#1{\left\|#1\right\|_{\infty}}
\def\nf#1{\left\|#1\right\|}

\def\la{\lambda}
\def\La{\Lambda}
\def\tp{^\top}

\newcommand \arrlf{\leftarrow}
\newcommand \arr{\rightarrow}

\def \Diag#1{\textbf{Diag}\left(#1\right)}

\newcommand{\zero}{\mathbf{0}}
\newcommand{\one}{\mathbf{1}}

\newcommand \inv{^{-1}}

\def\MatSize#1#2{\mathbb{R}^{#1\times#2}}
\newcommand \Real{\mathbb{R}}

\def \trace#1{\textbf{tr}\left(#1\right)}

\newcommand\YYhat{\widehat{\YY}}

\newcommand\na{\nabla}

\newcommand\alp{\alpha}

\newcommand\xa{\overline{\xx}}
\newcommand\ya{\overline{\yy}}

\newcommand\compress{\textsc{Compress}}

\newcommand{\mm}[2]{{\left\vert\kern-0.25ex\left\vert\kern-0.25ex\left\vert #2
    \right\vert\kern-0.25ex\right\vert\kern-0.25ex\right\vert}_{#1}}
\newcommand\mmb{{\big\vert\kern-0.25ex\big\vert\kern-0.25ex\big\vert}}
\newcommand\mmB{{\Big\vert\kern-0.25ex\Big\vert\kern-0.25ex\Big\vert}}
\def\mt#1{\mm{2}{#1}}
\def \mR#1{\mm{\rm R}{#1}}
\def \mC#1{\mm{\rm C}{#1}}
\def \nR#1{\nm{\rm R}{#1}}
\def \nC#1{\nm{\rm C}{#1}}
\newcommand\mmhalf{\vert\kern-0.25ex\vert\kern-0.25ex\vert}

\newcommand\rmR{{\rm R}}
\newcommand\rmC{{\rm C}}
\newcommand\PR{\PPi_{\rmR}}
\newcommand\PC{\PPi_{\rmC}}

\newcommand\almax{\widehat{\alpha}}
\newcommand\EY{\EE}
\def\sign#1{{\rm sign\pr{#1}}}

\newcommand\prt{\xi}
\newcommand\va{\widehat{\vv}}
\newcommand\vr{\overline{\vv}}
\newcommand\RL{\sigma}

\newcommand\degR{r}
\newcommand\gratemp{\dd}
\newcommand\AW{\widetilde{\La}}
\newcommand\RW{\widetilde{\La}_{\rmR}}
\newcommand\IW{\La}
\newcommand\RU{\overline{\RR}}
\newcommand\RP{\widehat{\RR}}
\newcommand\CQ{\widehat{\CC}}
\newcommand\Vara{V_1}
\newcommand\Varb{V_2}
\newcommand\Varc{V_3}
\newcommand\Vard{V_4}
\newcommand\Vare{V_5}
\newcommand\Varf{V_6}
\newcommand\ro{\calV}
\newcommand\alb{\widetilde{\alpha}}
\newcommand\vo{\widetilde{\vv}}
\newcommand\ve{\widehat{\vv}}
\newcommand\cb{\sigma'}
\newcommand\vb{\vr'}
\newcommand\Ab{\widetilde{\AA}}
\newcommand\pt{\psi}
\newcommand\gtil{\widetilde{g}}
\newcommand\db{\widetilde{\dd}}

\newcommand \Rb{\RR_{\beta}}
\newcommand \Cg{\CC_{\gamma}}
\newcommand \UR{\UU_{\rm R}}

\newcommand \EX{\WW}
\newcommand \eX{\ww}

\newcommand\PUSH{\textsc{PUSH}}
\newcommand\PULL{\textsc{PULL}}
\newcommand\Subroutine{\textbf{procedure }}
\newcommand\EndSubroutine{\textbf{end procedure }}
\newcommand\tabb{\hspace{0.4cm}}

\def \gF#1{\na \FF(\XX^{#1})}

\def\eq#1{\begin{align*}
             #1
          \end{align*}}

\newcommand\vecr{\ss_{\rm R}}
\newcommand\vecc{\ss_{\rm C}}

\ifCLASSINFOpdf
\else
\fi
\begin{document}

\title{Compressed Gradient Tracking for Decentralized Optimization Over General Directed Networks\thanks{{Lei Shi was partially supported by Shanghai Science and Technology Program [Project No. 20JC1412700 21JC1400600 and 19JC1420101] and National Natural Science Foundation of China (NSFC) under Grant No. 12171093. Shi Pu was supported by  Shenzhen Research Institute of Big Data (SRIBD) [Fund No.J00120190011] and  NSFC under Grant No. 62003287. Ming Yan was supported by  NSF grant DMS-2012439. Most work was done while the first author was visiting School of Data Science, The Chinese University of Hong Kong, Shenzhen. (Corresponding authors:
Shi Pu; Ming Yan.) }
		}
}
\author{Zhuoqing Song, Lei Shi, Shi Pu and Ming Yan
\thanks{Zhuoqing Song is with Shanghai Center for Mathematical Sciences, Fudan University, Shanghai, China. e-mail: zqsong19@fudan.edu.cn}
\thanks{Lei Shi is with School of Mathematical Sciences, Shanghai Key Laboratory for Contemporary Applied Mathematics, Fudan University, Shanghai, China. e-mail: leishi@fudan.edu.cn}
\thanks{Shi Pu is with School of Data Science, Shenzhen Research Institute of Big Data, The Chinese University of Hong Kong, Shenzhen, China. e-mail: pushi@cuhk.edu.cn}
\thanks{Ming Yan is with Department of Computational Mathematics, Science and Engineering and Department of Mathematics, Michigan State University, East Lansing, MI 48864 USA. e-mail: myan@msu.edu}
}

\markboth{Published on IEEE Transactions on Signal Processing, 70(2022), 1775--1787}%
{}
%



\maketitle
\etocdepthtag.toc{mtchapter}
\etocsettagdepth{mtchapter}{subsection}
\etocsettagdepth{mtappendix}{none}
\begin{abstract}
    In this paper, we propose two communication-efficient decentralized optimization algorithms over a  general directed multi-agent network.
    The first algorithm, termed Compressed Push-Pull (CPP), combines the gradient tracking Push-Pull method with communication compression.
    We show that CPP is applicable to a general class of unbiased compression operators and achieves linear convergence rate for strongly convex and smooth objective functions.
    The second algorithm is a broadcast-like version of CPP (B-CPP), and it also achieves linear convergence rate under the same conditions on the objective functions.
    B-CPP can be applied in an asynchronous broadcast setting and further reduce communication costs compared to CPP. Numerical experiments complement the theoretical analysis and confirm the effectiveness of the proposed methods.
\end{abstract}


%
\IEEEpeerreviewmaketitle

\section{Introduction}
%
%
%
%
%
%

\IEEEPARstart{I}{n} this paper, we focus on solving the decentralized optimization problem, where a system of $n$ agents, each having access to a private function $f_i(\xx)$, work collaboratively to obtain a consensual solution to the following problem:
\begin{equation}\label{eq: object}
  \min_{\xx\in\Real^p} f(\xx) = \frac{1}{n}\sum_{i=1}^n f_i(\xx),
\end{equation}
where $\xx$ is the global decision variable.
The $n$ agents are connected through a general directed network and only communicate directly with their immediate neighbors.
The problem (\ref{eq: object}) has received much attention in recent years due to its wide applications in distributed machine learning \cite{cohen2017projected,forrester2007multi,nedic2017fast}, multi-agent target seeking \cite{chen2012diffusion,pu2016noise}, and wireless networks \cite{baingana2014proximal,cohen2017distributed,mateos2012distributed}, among many others.
For example, the rapid development of distributed machine learning involves data whose size is getting increasingly large, and they are usually stored across multiple computing agents that are spatially distributed. Centering large amounts of data is often undesirable due to limited communication resources and/or privacy concerns,
and decentralized optimization serves as an important tool to solve such large-scale distributed learning problems due to its scalability, sparse communication, and better protection for data privacy~\cite{nedic2018distributed}.

Many methods have been proposed to solve the problem \eqref{eq: object} under various settings on the optimization objectives, network topologies, and communication protocols.
The paper~\cite{nedic2009distributed} developed a decentralized subgradient descent method (DGD) with diminishing stepsizes to reach the optimum for convex objective functions over an undirected network topology.
Subsequently, decentralized optimization methods for undirected networks, or more generally, with doubly stochastic mixing matrices, have been extensively studied in the literature; see, e.g.,~\cite{li2019decentralized,qu2017harnessing,scaman2018optimal,shi2015extra,uribe2017optimal,xu2015augmented}.
Among these works, EXTRA~\cite{shi2015extra} was the first method that achieves linear convergence  for strongly convex and smooth objective functions under symmetric stochastic matrices.
For directed networks, however, constructing a doubly stochastic mixing matrix usually requires a weight-balancing step, which could be costly when carried out in a distributed manner.
Therefore, the push-sum technique~\cite{kempe2003gossip} was utilized to overcome this issue.
Specifically, the push-sum based subgradient method in~\cite{nedic2014distributed} can be implemented over time-varying directed graphs, and linear convergence rates were achieved in~\cite{xi2017dextra, zeng2015extrapush} for minimizing strongly convex and smooth objective functions by applying the push-sum technique to EXTRA.

Gradient tracking is an important technique that has been successfully applied in many decentralized optimization algorithms.
Specifically, the methods proposed in~\cite{qu2017harnessing,nedic2017achieving,tian2018asy,xi2017add} employ gradient tracking to achieve linear convergence for strongly convex and smooth objective functions, where the work in \cite{nedic2017achieving,xi2017add,tian2018asy} particularly considered combining gradient tracking with the push-sum technique to accommodate directed graphs.
The methods can also be applied to time-varying graphs~\cite{nedic2017achieving} and asynchronous settings~\cite{tian2018asy}.
The Push-Pull/$\calA\calB$ method introduced in~\cite{pu2020push, xin2018linear} modified the gradient tracking methods to deal with directed network topologies without the push-sum technique.
The algorithm uses a row stochastic matrix to mix the local decision variables and a column stochastic matrix to mix the auxiliary variables that track the average gradients over the network. It can unify different network architectures, including peer-to-peer, master-slave, and leader-follower architectures \cite{pu2020push}.
For minimizing strongly convex and smooth objectives, the Push-Pull/$\calA\calB$ method not only enjoys linear convergence over fixed graphs~\cite{pu2020push, xin2018linear}, but also works well under time-varying graphs and asynchronous settings~\cite{pu2020push, saadatniaki2020decentralized, zhangfully}.

In decentralized optimization, efficient communication is critical for enhancing algorithm performance and system scalability. One major approach to reduce communication costs is considering communication compression, which is essential especially under limited communication bandwidth.
Recently, several compression methods have been proposed for distributed and federated learning, including~\cite{alis, bernstein2018signsgd,beznosikov2020biased, karimireddy2019error, liu2020distributed, liu2020double,  mishchenko2019distributed, seide20141, stich2020communication, stich2018sparsified, tang2019doublesqueeze,xu2020compressed,lin2017deep}.
{Recent works have tried to combine the communication compression methods with decentralized optimization. 
The existence of compression errors may result in inferior convergence performance compared to uncompressed or centralized algorithms. For example, the methods considered by~\cite{xiao2005scheme,carli2007average,nedic2008distributed,aysal2008distributed,carli2010gossip,yuan2012distributed} can only guarantee to reach a neighborhood of the desired solutions when the compression errors exist. 
QDGD~\cite{reisizadeh2019exact} achieves a vanishing mean solution error with a slower rate than the centralized method.   
To reduce the error from compression, some works~\cite{carli2010quantized,doan2018accelerating,berahas2019nested} increase compression accuracy as the iteration grows to guarantee the convergence. However, they still need high communication costs to get highly accurate solutions. Techniques to remedy this increased communication costs include gradient difference compression~\cite{mishchenko2019distributed,li2020acceleration,liu2020linear} and error compensation~\cite{stich2018sparsified,koloskova2019decentralized,koloskova2019decentralizeda}, which enjoy better performance than direct compression. 
In~\cite{tang2018communication}, the difference compression (DCD-PSGD) and extrapolation compression (ECD-PSGD) algorithms were proposed to reach the same convergence rate of the centralized schemes with additional requirements on the compression ratio.  In~\cite{koloskova2019decentralizeda}, a quantized decentralized SGD (CHOCO-SGD) method was proposed 
and shown to converge to the optimal solution or a stationary point at a comparable rate as the centralized SGD method.   
The works~\cite{koloskova2019decentralized} and~\cite{tang2019deepsqueeze} 
also achieve comparable convergence rates with the centralized scheme for solving nonconvex problems.}


For strongly convex and smooth objective functions, \cite{kajiyama2020linear} first considered a linearly convergent gradient tracking method based on a specific quantizer.
More recently, the paper \cite{liu2020linear} introduced LEAD that works with a general class of compression operators and still enjoys linear convergence. Some recent developments can be found in \cite{li2021compressed,xiong2021quantized,zhang2021innovation}, where \cite{xiong2021quantized} particularly combined Push-Pull/$\calA\calB$ with a special quantizer to achieve linear convergence over directed graphs.

In this paper, we consider decentralized optimization over general directed networks and propose a novel Compressed Push-Pull method (CPP) that combines Push-Pull/$\calA\calB$ with a general class of unbiased compression operators. CPP enjoys large flexibility in both the compression method and the network topology. We show CPP achieves linear convergence rate under strongly convex and smooth objective functions.

Broadcast or gossip-based communication protocols are popular choices for distributed computation due to their low communication costs \cite{aysal2009broadcast,boyd2006randomized,pu2020distributed}.
In the second part of this paper, we propose a broadcast-like CPP algorithm (B-CPP) that allows for asynchronous updates of the agents: at every iteration of the algorithm, only a subset of the agents wake up to perform prescribed updates. Thus, B-CPP is more flexible, and due to its broadcast nature, it can further save communication over CPP in certain scenarios \cite{pu2020distributed}. We show that B-CPP also achieves linear convergence for minimizing  strongly convex and smooth objectives.

The main contributions of this paper are summarized as follows.
\begin{itemize}
    \item We propose CPP -- a novel decentralized optimization method with communication compression. The method works under a general class of compression operators and is shown to achieve linear convergence for strongly convex and smooth objective functions over general directed graphs. To the best of our knowledge, CPP is the first method that enjoys linear convergence under such a general setting.

\item
We consider an asynchronous broadcast version of CPP (B-CPP). B-CPP further reduces the communicated data per iteration and is also provably linearly convergent over directed graphs for minimizing strongly convex and smooth objective functions. Numerical experiments demonstrate the advantages of B-CPP in saving communication costs.

\end{itemize}

The rest of this paper is organized as follows. We provide necessary notation and assumptions in Section \ref{sec:notation}. CPP is introduced and analyzed in Section \ref{sec:CPP}. In Section \ref{sec:BCPP}, we consider the algorithm B-CPP. Numerical examples are presented in Section \ref{sec:numerical}, and we conclude the paper in Section \ref{sec:conclusions}. 

\section{Notation and Assumptions}
\label{sec:notation}
%
Denote the set of agents as $\calN = \dr{1, 2, \cdots, n}$. At iteration $k$, each agent $i$ has a local estimation $\xx_i^k \in \Real^p$ of the global decision variable and an auxiliary variable $\yy_i^k$. We use lowercase bold letters to denote the local variables, and their counterpart uppercase bold letters denote the concatenation of these local variables. For instance, $\XX^k,~\na\FF(\XX^k)$ are the concatenation of $\dr{\xx_i^k}_{i\in \calN},~\dr{\na f_i(\xx_i^k)}_{i\in \calN}$, respectively, and their connections are
\begin{align*}
  &\XX^k = \big[\xx_1^k, \cdots, \xx_n^k\big]\tp \in \MatSize{n}{p},\\
  &\na\FF(\XX^k) = \big[\na f_1(\xx_1^k), \cdots, \na f_n(\xx_n^k)\big]\tp \in \MatSize{n}{p}.
\end{align*}

\begin{assumption}\label{assp: require each fi}
    Each $f_i$ is $\mu$-strongly convex ($\mu>0$) and $L$-smooth, i.e., for any $\xx_1, \xx_2\in \Real^p$,
    \seq{
      &\jr{\xx_1 - \xx_2, \na f_i(\xx_1) - \na f_i(\xx_2)} \geq \mu \nt{\xx_1 - \xx_2}^2, \\
      &\nt{\na f_i(\xx_1) - \na f_i(\xx_2)} \leq L \nt{\xx_1 - \xx_2}.
    }
\end{assumption}

%

Since all $f_i(\xx)$ are strongly convex, $f(\xx)$ admits a unique minimizer $\xx^*$. Denote $\XX^* = \one\xx^{*\top}, $ where $\one$ is the all-ones column vector.

Given any nonnegative matrix $\MM$, we denote by $\calG_{\MM}$ the induced graph by $\MM$. The sets $\calN_{\MM, i}^- \defeq \dr{j\in\calN: \MM_{ij} > 0}$ and $\calN_{\MM, i}^+ \defeq \dr{j\in\calN: \MM_{ji}>0}$ are called the in-neighbors and out-neighbors of agent $i$.

The communication between all the agents is modeled by directed graphs. Given a strongly connected graph $\calG = \pr{\calN, \calE}$ with $\calE \subset \calN \times \calN$ being the edge set, agent $i$ can receive information from agent $j$ if and only if $(i,j)\in \calE$. There are two $n$-by-$n$ nonnegative matrices $\RR$ and $\CC$.
A spanning tree $\calT$ rooted at some $i\in \calN$ in $\calG_{\RR}$ is a subgraph of $\calG_{\RR}$ containing $n-1$ edges, and each node except $i$ can be connected to $i$ by a path in $\calT$.
Let $\calR_{\RR}, \calR_{\CC\tp}$ denote the roots of the spanning trees in $\calG_{\RR}$ and $\calG_{\CC\tp}$.  
We have the following assumption on $\RR$ and $\CC$.
\newcommand\altil{\overline{\alpha}}
\begin{assumption}\label{asp: R and C}
    The matrix $\RR$ is supported by $\calG$, i.e., $\calE_{\RR} = \dr{(i,j)\in\calN\times\calN\Big| \RR_{ij} > 0 } \subset \calE$, and $\RR$ is a row stochastic matrix, i.e., $\RR\one = \one$. The matrix $\CC$ is also supported by $\calG$, and $\CC$ is a column stochastic matrix, i.e., $\one\tp\CC = \one\tp$.
    In addition, $\calR_{\RR} \cap \calR_{\CC\tp} \neq \emptyset$.
\end{assumption}
By \cite[Lemma~1]{pu2020push}, Assumption~\ref{asp: R and C} implies the following facts: {$\RR$ has a unique left eigenvector $\vecr$ with respect to $1$ such that $\vecr^\top\one=n$.
$\CC$ has a unique right   eigenvector $\vecc$ with respect to $1$ such that $\vecc\tp\one=n$. }
In addition, the entries of $\vecr$ and $\vecc$ are all nonnegative. All nonzero entries of $\vecr$, $\vecc$ correspond to $\calR_{\RR}$, $\calR_{\CC\tp}$, respectively.
Because $\calR_{\RR} \cap \calR_{\CC\tp} \neq \emptyset$, we have $\vecr\tp\vecc > 0$.

We denote the spectral radius of matrix $\AA$ as $\rho\pr{\AA}$. The inner product of two matrices is defined as $\jr{\XX, \YY} = \trace{\XX\tp\YY}$, and the Frobenius norm is $\nf{\XX}_{\rm F} = \sqrt{\jr{\XX, \XX}}$.
Given a vector $\dd$, $\dd_{a:b}$ is the subvector of $\dd$ containing the entries indexed from $a$ to $b$.
Given a matrix $\AA$, the notion $\AA_{a:b,c:d}$ denotes the submatrix containing the entries with row index in $[a, b]$ and column index in $[c, d]$.
We abbreviate ``$1:n$" by ``$:$" and ``$i:i$" is abbreviated as ``$i$".
Especially in our notations, $\AA_{i,:}$ and $\AA_{:, j}$ denote the $i$-th row and $j$-th column of $\AA$, respectively.
{For vector $\vv\in \Real^m$, $\Diag{\vv}$ is an $m$-by-$m$ diagonal matrix with $\Diag{\vv}_{ii} = \vv_i$. }

\begin{definition}\label{def: induced matrix norm n by p}
Given a vector norm $\nm{*}{\cdot}$, we define the corresponding matrix norm on an $n\times p$ matrix $\AA$ as
\seq{
\mm{*}{\AA} = \nt{\Big[\nm{*}{\AA_{:, 1}}, \nm{*}{\AA_{:, 2}}, \cdots, \nm{*}{\AA_{:, p}}\Big]}.
}
When $\nm{*}{\cdot}=\|{\cdot}\|_2$, we have $\mm{2}{\AA}=\|{\AA}\|_{\rm F}$, the Frobenius norm .
\end{definition}

\begin{definition}\label{def: induced matrix norm n by n}
    Given a vector norm $\nm{*}{\cdot}$, we define the corresponding induced norm on an $n\times n$ matrix $\AA$ as
    $
      \nm{*}{\AA} = \sup_{\xx \neq 0} \frac{\nm{*}{\AA\xx}}{\nm{*}{\xx}}.
    $
\end{definition}

\begin{lemma}[Lemma 5 in \cite{pu2020push}]\label{eq: basic calculus of mm}
For any matrices $\AA\in\MatSize{n}{p}$, $\WW\in\MatSize{n}{n}$, and a vector norm $\nm{*}{\cdot}$, we have $\mm{*}{\WW\AA} \leq \nm{*}{\WW} \mm{*}{\AA}$.
For any vectors $\aa\in \Real^n$, $\bb\in \Real^p$, $\mm{*}{\aa\bb\tp} \leq \nm{*}{\aa} \nt{\bb}$.
\end{lemma}

\def\cerr#1{C_{\rm #1}}
The compression in this paper is denoted by $\compress\pr{\cdot}$. For any matrix $\XX\in \MatSize{n}{p}$, we denote $\compress(\XX)$ as an $n$-by-$p$ matrix with the $i$-th row being $\compress\pr{\XX_{i,:}}$.  
\begin{assumption}\label{assp: compress}
The compression is unbiased, i.e.,
    given $\xx\in\mathbb{R}^{p}$ and $\xxhat = \compress\pr{\xx} $, there exists $C_2 > 0$ such that
    $
      \mathbb{E}\left[\xxhat\big|\xx\right] = \xx,\ \text{and \ }      \expec{\nt{\xxhat-\xx}^2\big|\xx} \leq \cerr{2} \nt{\xx}^2.
    $
    And the random variables generated inside the procedure $\compress\pr{\xx}$ depends on $\xx$ only.
\end{assumption}

\section{A Push-Pull Method with Compression}
\label{sec:CPP}

In this section, we propose a Push-Pull method with Compression (CPP) in Algorithm~\ref{alg: push pull with compression}.
\begin{algorithm}
\caption{  Compressed Push-Pull \ (in agent $i$ 's perspective) }
\label{alg: push pull with compression}
\KwIn{initial position $\xx_i^0$,
     stepsize $\alpha_i$,
     averaging parameters $\beta, \gamma, \eta \in (0, 1] $,
     network information $\RR_{ij} (j\in \calN_{{\rm R}, i}^-)$, $\CC_{ij} (j\in \calN_{{\rm C}, i}^-)$ and the sets $\calN_{\rmR, i}^+$, $\calN_{{\rm C}, i}^+$,
    total iteration number $K$
}
\KwOut{$\xx_i^K $}

Initiate $\yy_i^0 = \na f_i(\xx_i^0), \uu_i^0 = \uu_{\rmR, i}^0  = \zero$. \;
\For{$k=0$ to $K - 1$}{
    Call procedure $\PULL$: $\pr{\xxhat_{{\rm R}, i}^k, \uu_i^{k+1}, \uu_{{\rm R}, i}^{k+1}} = \PULL\pr{\xx_i^k, \uu_i^k, \uu_{{\rm R}, i}^k}$. \;
    Update
      $\xx_i^{k+1} = (1 - \beta)\xx_i^k + \beta \xxhat_{{\rm R}, i}^k - \alpha_i \yy_i^k. $
    \;

    Call procedure $\PUSH$: $\pr{\yyhat_i^k, \yyhat_{{\rm C}, i}^k} = \PUSH\pr{\yy_i^k}$. \;
    Compute local gradient $\na f_i\pr{\xx_i^{k+1}} $ . \;
    Update $\yy_i^{k+1} = \yy_i^k + \gamma\yyhat_{\rmC, i}^k - \gamma\yyhat_i^k + {\na f_i\pr{\xx_i^{k+1}} - \na f_i\pr{\xx_i^k}}.   $
}
\vspace{0.2cm}
\Subroutine $\PULL\pr{\xx_i, \uu_i, \uu_{{\rm R}, i}}$ \;
    \tabb $ \pp_i = \compress\pr{\xx_i - \uu_i}.  $ 
    \;
    \tabb $\pp_{{\rm R}, i} = \sum_{j\in \calN_{{\rm \RR}, i}^-} \RR_{ij} \pp_j $ .  
    \;
    \tabb $\xxhat_i = \uu_i + \pp_i  $ .  \;
    \tabb $\xxhat_{\rmR, i} = \uu_{\rmR, i} + \pp_{\rmR, i}  $.  \;
    \tabb $\uu_i \arrlf \pr{1 - \eta}\uu_i + \eta \xxhat_i$ .  \;
    \tabb $\uu_{{\rm R}, i} \arrlf \pr{1-\eta}\uu_{{\rm R}, i} + \eta \xxhat_{{\rm R}, i}$ . \;
    \tabb \textbf{Return:} $\xxhat_{{\rm R}, i}, \uu_i, \uu_{{\rm R}, i}$.  \;
\EndSubroutine \;
\vspace{0.2cm}
\Subroutine $\PUSH\pr{\yy_i }$ \;
    \tabb  $ \yyhat_i = \compress\pr{\yy_i }. $ 
    \;
    \tabb $\yyhat_{{\rm C}, i} = \sum_{j\in \calN_{{\rm \CC}, i}^- } \CC_{ij} \yyhat_j $ .  
    \;
    \tabb \textbf{Return:} $\yyhat_i, \yyhat_{{\rm C}, i} $. \;
\EndSubroutine
\end{algorithm}
{We start from discussing the following scheme of Push-Pull/$\calA\calB$ from the viewpoint of agent $i$~\cite{pu2020push, xin2018linear}:
\begin{equation*}
    \left\{
  \begin{aligned}
      & \xx_i^{k+1} = \sum_{j\in \calN_{\rm R,i}^-}\RR_{ij}\xx_j^k - \alpha_i \yy_i^k \\
      & \yy_i^{k+1} = \sum_{j\in \calN_{\rm C,i}^-}\CC_{ij}\yy_j^k + \na f_i\pr{\xx_i^{k+1}} - \na f_i\pr{\xx_i^k}
  \end{aligned}
  \right.
\end{equation*}
and  $\yy_i^0 = \na f_i\pr{\xx_i^0}$.
At each iteration, agent $i$ computes $\sum_{j\in \calN_{\rm R, i}^-}\RR_{ij}\xx_j^k$ to average the local decision variables received from its neighbors.
As each agent takes such a step, the local decision variables tend to get closer with each other and eventually reach consensus when $\yy_i^k$ goes to zero.
In the next ``gradient tracking" step,
by adding the gradient difference term $\na f_i\pr{\xx_i^{k+1}} - \na f_i\pr{\xx_i^k} $ into $\yy_i^{k+1}$ and considering that $\yy_i^0 = \na f_i\pr{\xx_i^0}$ and $\one\tp\CC = \one\tp$,
we have  $\sum_{i\in \calN}\yy_i^k = \sum_{i\in\calN}\na f_i\pr{\xx_i^k}$
by induction, which indicates that $\dr{\yy_i^k}_{i\in \calN}$ can help track the average gradients over the network.
}

{The Compressed Push-Pull (CPP) method differs from Push-Pull/$\calA\calB$ mainly in the averaging step which requires communication.
To alleviate the impact of compression errors, we use a ``damped" averaging step and replace the exact local variables by their communication-efficient version, i.e., $\pr{1 - \beta}\xx_i^k + \beta\sum_{j\in \calN_{\rm R, i}^-}\RR_{ij}\xxhat_j^k$.
For $\yy_i^k$, we take a slightly different averaging step  so that the relation $\sum_{i\in \calN}\yy_i^k = \sum_{i\in\calN}\na f_i\pr{\xx_i^k}$ is preserved and the compression errors can also be bounded ``safely".}

{However, sending $\xxhat_i^k$ itself can still be expensive.
Therefore, we call the procedure $\PULL$ to save the communication costs.
In this procedure, $(\uu_i,\uu_{{\rm R},i})$ is used to track $(\xx_i, \sum_{j\in \calN_{\rm R, i}^-} \RR_{ij}\xx_j )$.
The compression and the communication are applied on the difference $(\xx_i-\uu_i)$ and its compressed version, respectively.
It can be derived from the relation $\pp_{{\rm R}, i} = \sum_{j\in \calN_{{\rm \RR}, i}^-} \RR_{ij} \pp_j $ and induction that
in each call to $\PULL$,
$
 \uu_{\rm R, i} = \sum_{j\in \calN_{\rm R, i}^-}\RR_{ij}\uu_j.
$
And the compression error could be very small if the difference $\xx_i-\uu_i$ is small. 
With this observation, we have
$
    \xxhat_{\rm R, i} = \sum_{j\in \calN_{\rm R, i}^-}\RR_{ij}\xxhat_j
$
with $\xxhat_j$ being an unbiased approximation of $\xx_j$, whose variance converges to 0. }

{In the $\PUSH$ procedure, since the variable $\yy_i^k$ converges to zero as we will show later, we can simply compress it and estimate $\sum_{j\in \calN_{\rm C, i}^-}\CC_{ij}\yy_j^k$ using the compressed values.  Similarly, we have that in each call to $\PUSH$, $\widehat{\yy}_{\rmC, i} = \sum_{j\in \calN_{\rm C, i}^-}\CC_{ij}\widehat{\yy}_j$, with $\widehat{\yy}_j$ being the unbiased compression of $\yy_j$, whose variance converges to 0 as the input $\yy^k_i$ converges to $\zero$.
}

{Let
$\almax = \max\limits_{1\leq i\leq n} \alpha_i,
$
and define
$
  \aalpha
$ as a diagonal matrix with $\aalpha_{ii} = \alpha_i$.
}
Then, we can rewrite Algorithm~\ref{alg: push pull with compression} into the following matrix form
\begin{equation}\label{eq: alg push pull with compress origin matrix form}
 \left\{\small
  \begin{aligned}
    &\XXhat^k = \UU^k + \compress\pr{\XX^k - \UU^k},  \\
    &\YYhat^k = \compress\pr{ \YY^k  },   \\
    &\XX^{k+1} = \pr{1 - \beta}\XX^k + \beta\RR\XXhat^k - \aalpha\YY^k,   \\
    &\YY^{k+1} = \YY^k + \gamma\pr{\CC - \II}\YYhat^k  \\ & \qquad\quad\quad + \nabla\FF(\XX^{k+1}) - \nabla\FF(\XX^k),  \\
    &\UU^{k+1} = \pr{1 - \eta}\UU^k + \eta\XXhat^k .
  \end{aligned}
  \right.
\end{equation}

Comparing to the standard Push-Pull method in~\cite{pu2020push}, the procedures $\PUSH$ and $\PULL$ save communication costs at the cost of error in the mixing. Note $\XXhat^k$ and $\YYhat^k$ are approximations for $\XX^k$ and $\YY^k$, respectively. The difference between $\XX^k$ (or $\YY^k$) and $\XXhat^k$ (or $\YYhat^k$) are induced by the compression errors. Let us denote the approximation error as $\eX_i^k = \xx_i^k - \xxhat_i^k$ and $\ee_i^k = \yy_i^k - \yyhat_i^k$. We further write them into the following matrix form
\seql{\label{eq: error in matrix form}}{
  \EX^k = \XX^k - \XXhat^k,\ \EE^k = \YY^k - \YYhat^k.
}
From the $\PULL$ procedure, we can see that $\EX^k$ is also the compression error for $\XX^k-\UU^k$. The update of $\UU^{k+1} = \pr{1 - \eta}\UU^k + \eta\XXhat^k$ indicates that $\UU^k$ is tracking the motions of $\XX^k$. As $\UU^k$ approaches $\XX^k$, by Assumption~\ref{assp: compress}, the variance of the approximation error $\EX^k$ will also tends to $\zero$.

Since $\YY^0 = \na \FF\pr{\XX^0}$ and $\one\tp\CC = \one\tp$, it follows from (\ref{eq: Y update k+1}) and induction that
\seql{\label{eq: YY track gradient}}{\one\tp\YY^k = \one\tp\na\FF\pr{\XX^k}.}
{Then, as the local variables  become closer to each other,
$\yy_i^k$ are more and more parallel to $\one\tp\gF{k} $.
As $\XX^k$ approaches the optimal point, $\one\tp\gF{k}$ tends to $\one\tp\na \FF\pr{\XX^*} = \zero\tp$.} This indicates that $\YY^k$ also tends to $\zero$. Then, by Assumption~\ref{assp: compress}, the variance of compression error $\EY^k$ will also tend to $\zero$. As the compression errors $\EX^k$, $\EY^k$ tend to $\zero$, we will have $\XXhat^k \approx \XX^k$ and $\YYhat^k \approx \YY^k$. We remark that if $\XXhat^k = \XX^k$, $\YYhat^k = \YY^k$, Algorithm~\ref{alg: push pull with compression} will reduce to the PUSH-PULL method in~\cite{pu2020push}.

To show the convergence of the proposed compressed algorithm, we define
$
     \Rb = (1 - \beta)\II + \beta \RR,\
     \Cg = (1 - \gamma)\II + \gamma \CC,
$
and rewrite (\ref{eq: alg push pull with compress origin matrix form}) as
\begin{subequations}\label{eq: push pull with compress iteration}
  \small
    \begin{numcases}{}
     \EX^k = \XX^k - \UU^k - \compress\pr{\XX^k - \UU^k},  \\
     \EY^k  = \YY^k - \compress\pr{\YY^k},  \\
     \XX^{k+1} = \Rb \XX^k - \aalpha \YY^k - \beta \RR \EX^k , \label{eq: X update k+1}   \\
     \YY^{k+1} = \Cg \YY^k + \na \FF(\XX^{k+1}) - \na \FF(\XX^k) \notag  \\
     \quad\quad\quad\quad + \gamma\pr{\II - \CC} \EE^k , \label{eq: Y update k+1} \\
     \UU^{k+1} = (1 - \eta)\UU^k + \eta \XX^k - \eta \EX^k, \label{eq: U update k+1}
    \end{numcases}
\end{subequations}
We define the value $\altil = {\vecr\tp\aalpha\vecc}/{n}$.
Under Assumption~\ref{asp: R and C}, $\altil > 0$ can be satisfied by setting $\alp_i > 0$ for at least one $i\in \calR_{\RR}\cap\calR_{\CC\tp}$.
When $\altil > 0$,
let us define a constant $w \geq \almax/\altil \geq \frac{n}{\vecr\tp\vecc}$.
Note that the analysis in the rest of this section holds true for arbitrary $w \geq \almax/\altil$.
As we will choose $w \geq \almax/\altil$ as an auxiliary parameter to help choose proper stepsizes $\dr{\alp_i}_{i\in\calN}$ in Theorem~\ref{thm:linconv} which is the main theorem of this section,
we use the same $w$ in the analysis below.

Let $\calF_k$ denote the $\sigma$-field generated by $\dr{\EE^j, \EX^j}_{j=0}^{k-1}$, and $\calF^+_k$ is the $\sigma$-field generated by $\dr{\EE^j, \EX^j}_{j=0}^{k-1}\cup\dr{\EX^k}$.
Thus, we have $\calF_1 \subset \calF^+_1 \subset \calF_2 \subset \calF^+_2 \subset \cdots \subset \calF_k \subset \calF^+_k \subset \cdots $.
By Lemma~\ref{lem: condi expec} below, $\calF_k = \sigma\pr{\XX^i, \YY^i, \UU^i: 0\leq i < k}$ and $\calF^+_k = \sigma\pr{\XX^i, \UU^i, \YY^j : 0\leq i\leq k, 0\leq j < k}.$

The following lemma comes from (\ref{eq: X update k+1})-(\ref{eq: U update k+1}) and Assumption~\ref{assp: compress}.
\begin{lemma}\label{lem: condi expec}
The variables $\XX^k, \YY^k, \UU^k, \VV^k$ are measurable with respect to $\calF_k$, and $\XX^k$ is measurable with respect to $\calF^+_{k-1}$.
Moreover, we have
\seql{\label{eq: EX EE condi expec zero}}{
\expec{\EX^k | \calF_k} = \expec{ \EE^k | \calF^+_k } = \zero .}
\end{lemma}
\begin{proof}
    By expanding (\ref{eq: X update k+1})-(\ref{eq: U update k+1}) recursively, $\XX^k$, $\YY^k$, $\UU^k$, and $\VV^k$ can be represented by linear combinations of $\XX^0, \YY^0, \UU^0, \VV^0$ and random variables $\dr{\EE^j, \EX^j}_{j=0}^{k-1}$, i.e., $\XX^k, \YY^k, \UU^k, \VV^k$ are measurable with respect to $\calF_k$.
    By (\ref{eq: X update k+1}), $\XX^k$ is a linear combination of $\EX^{k-1}$ and $\calF_{k-1}$-measurable variables $\XX^{k-1}, \YY^{k-1}$, thus, $\XX^k$ is measurable with respect to $\calF^+_{k-1}$.
    And we obtain (\ref{eq: EX EE condi expec zero}) directly from Assumption~\ref{assp: compress}.
\end{proof}

\subsection{Convergence Analysis for CPP}
In this section, we analyze the convergence rates of Algorithm~\ref{alg: push pull with compression}.
To begin with, we define the averages of $\XX^k$, $\YY^k$ as follows,
\seql{\label{eq: consensus variable def}}{
  \xa^k = \frac{1}{n}\vecr\tp\XX^k,\ \ya^k = \frac{1}{n}\one\tp\YY^k.}
We define two matrices
$
    \PR = \II - \frac{\one\vecr\tp}{n},\ \PC = \II - \frac{\vecc\one\tp}{n}.
$
It follows from direct calculation that
$
    \PR\RR = \RR\PR = \RR - \frac{\one\vecr\tp}{n},
$
$
    \PC\CC = \CC\PC = \CC  - \frac{\vecc\one\tp}{n},
$
$
    \PR\XX^k = \XX^k - \one\xa_k,
$
$
    \PC\YY^k = \YY^k - \vecc\ya_k,
$
$
    \pr{\RR - \II}\PR = \RR - \II,
$
and
$
    \PC\pr{\II - \CC} = \II - \CC.
$

In the analysis below,
$\expec{\nt{\xa^k - \xx^*}^2}$ is employed to measure the closeness to the optimal point; $\expec{\mR{\PR\XX^k}^2}$ and $\expec{\mC{\PC\YY^k}^2}$ measure the consensus error and the gradient tracking error, respectively.
{
To bound the compression errors, we use
$\expec{\mt{\UU^k - \XX^k}^2}$ to measure the convergence of the momentums.
}
The matrix norms $\mR{\cdot}$ and $\mC{\cdot}$ are defined in Lemma~\ref{def nR nC}.

{Compared with the convergence analysis for Push-Pull, the analysis for CPP additionally requires dealing with the compression errors and establishing the relationship between the error terms of different types. Moreover, another term $\expec{\mt{\YY^k}^2}$ is considered to simplify the proof, and its role will be made clear in the follow-up analysis.}

Now, we have the following expansion by multiplying $\frac{\vecr\tp}{n}$ on both sides of (\ref{eq: X update k+1})
\seql{\label{eq: xa expansion}}{
     &\xa^{k+1} = \frac{1}{n}\vecr\tp\pr{\Rb \XX^k - \aalpha \YY^k - \beta \RR \EX^k}  \\
       =& \xa^k - \frac{1}{n}\vecr\tp\aalpha\pr{\YY^k - \vecc\ya^k + \vecc\ya^k} - \frac{\beta}{n} \vecr\tp \EX^k \\
       =& \xa^k - \altil \ya^k - \frac{1}{n}\vecr\tp\aalpha\PC\YY^k - \frac{\beta}{n} \vecr\tp \EX^k \\
       =& \xa^k - \altil \gg^k + \altil\pr{\gg^k - \ya^k} - \frac{1}{n}\vecr\tp\aalpha\PC\YY^k \\
        & - \frac{\beta}{n} \vecr\tp \EX^k,
}
where
$
  \gg^k = \na f(\xa^k).
$

Multiplying $\PR$ on both sides of (\ref{eq: X update k+1}), we have
\seql{\label{eq: x - one xa expansion}}{
     &\PR\XX^{k+1} \\
       =& \PR\Rb \PR\XX^k - \PR\aalpha \YY^k
       - \beta \PR\RR\EX^k ,
}
Multiplying $\PC$ on both sides of (\ref{eq: Y update k+1}), we obtain
\seql{\label{eq: y - one ya expansion}}{
     \PC\YY^{k+1}
       =& \PC\Cg \PC\YY^k  + \gamma \pr{\II - \CC} \EE^k\\
       &+ \PC \pr{\na \FF\pr{\XX^{k+1}} - \na \FF\pr{\XX^k}}.
}

\newcommand \theR{\theta_{\rm R}}
\newcommand \theC{\theta_{\rm C}}
\newcommand \Rtil{\widetilde{\RR}}
\newcommand \Ctil{\widetilde{\CC}}
\newcommand\traRt{\delta_{{\rm R}, 2}}
\newcommand\tratR{\delta_{2, {\rm R}}}
\newcommand\traCt{\delta_{{\rm C}, 2}}
\newcommand\tratC{\delta_{2, {\rm C}}}
\newcommand\traRC{\delta_{{\rm R}, {\rm C}}}
\newcommand\traCR{\delta_{{\rm C}, {\rm R}}}

\begin{lemma}\label{def nR nC}
    There are invertible matrices $\Rtil, \Ctil \in \MatSize{n}{n}$ inducing vector norms $\nR{\xx} \defeq \nt{\Rtil\xx}$, $\nC{\xx} \defeq \nt{\Ctil\xx}$ for $\xx \in \Real^n$.
    The corresponding matrix norms $\nR{\cdot}, \nC{\cdot}$ defined by Definition~\ref{def: induced matrix norm n by n} satisfy:
    for any $\beta, \gamma\in [0, 1]$,
    \seql{\label{eq: nR Rb nC Cg}}{
      \nR{\PR\Rb} \leq 1 - \theR \beta,\
      \nC{\PC\Cg} \leq 1 - \theC \gamma,
    }
    where $\theR, \theC $ are constants in $(0, 1]$.
    Especially, $\nR{\PR} = \nC{\PC} = 1$ and $\nR{\RR} \leq 2$, $\nR{\RR - \II} \leq 3$, $\nC{\CC} \leq 2$, $\nC{\CC - \II} \leq 3$.
    Additionally, $\nt{\xx} \leq \nR{\xx}$, $\nt{\xx} \leq \nC{\xx}$, $\forall \xx\in \Real^n$;
    $\mt{\AA} \leq \mR{\AA}  $, $\mt{\AA} \leq \mC{\AA} $, $\forall \AA\in \MatSize{n}{p}$.
    There exist constants $\traRt $, $\traCt $ such that $\nR{\xx} \leq \traRt \nt{\xx}$, $\nC{\xx} \leq \traCt\nt{\xx}$, $\forall \xx\in \Real^n$.
\end{lemma}
\begin{proof}
    See Supplementary Material~\ref{lem: proof of def nR nC}.
\end{proof}
In the following, $\nR{\cdot}$, $\nC{\cdot}$, $\mR{\cdot}$, and  $\mC{\cdot}$ refer to the norms defined in the above lemma.

\begin{lemma}\label{lem: Hilbert mm AM GM}
  For any vector norm $\nm{*}{\cdot}$ induced by the inner product $\jr{\cdot, \cdot}_*$,
  and for any $\theta > 1$, we have
  \seq{
      \mm{*}{\AA + \BB}^2 \leq \theta \mm{*}{\AA}^2 + \frac{\theta}{\theta - 1 } \mm{*}{\BB}^2,\ \forall \AA, \BB\in \MatSize{n}{p}  .
  }
  Especially, $\nm{*}{\cdot}$ can be taken as $\nt{\cdot}$, $\nR{\cdot}$, $\nC{\cdot}$.
\end{lemma}
\begin{proof}
    By Definition~\ref{def: induced matrix norm n by p}, it suffices to prove $\nm{*}{\aa+\bb}^2 \leq \theta\nm{*}{\aa}^2 + \frac{\theta}{\theta - 1}\nm{*}{\bb}^2$ for any vectors $\aa, \bb\in \Real^n$.
    And it can be verified by the elementary inequality $2\jr{\aa, \bb}_* \leq \pr{\theta - 1}\nm{*}{\aa}^2 + \frac{1}{\theta - 1}\nm{*}{\bb}^2 $.
    The vector norms $\nR{\cdot}, \nC{\cdot}$ are induced by inner products $\jr{\Rtil\xx, \Rtil\yy}_*, \jr{\Ctil\xx, \Ctil\yy}_*$, respectively.
\end{proof}

The next lemma uses inequalities to derive recursive bounds for the quantities $\expec{\nt{\xa^k - \xx^*}^2}, \expec{\mR{\PR\XX^k}^2}, \expec{\mC{\PC\YY^k}^2}$ and $\expec{\mt{\UU^k - \XX^k}^2}$.
The quantity $\expec{\mt{\YY^k}^2}$ is introduced to help simplify the proof.
\begin{lemma}\label{lem:cpushpullmatineq}
    For $k \geq 0$, take $\dd^k \in \Real^5$ as
    \seq{
        \dd^{k}=\mathbb{E}\bigg[
        \Big(&{\nt{\xa^k - \xx^*}^2}, {\mR{\PR\XX^k}^2}, \\
        &{\mC{\PC\YY^k}^2},
        {\mt{\UU^k - \XX^k}^2}, {\mt{\YY^k}^2}  \Big)\bigg]\tp .
    }
    Define
    a parameterized matrix $\AA\pr{\almax, \beta, \gamma, \eta}  $ as follows
    \seql{\label{eq:defA1}}{
      \AA_{:, 1:3} &=
      \begin{pmatrix}
        1 - c_1\almax & c_2\almax & c_3\almax  \\
        0 & 1 - \theR\beta & 0   \\
        0 & \frac{8  n c_7\beta^2}{\gamma} & 1 - \theC\gamma   \\
        0 & \frac{8  n \beta^2}{\eta} & 0   \\
         c_{10} & c_9 & 3
      \end{pmatrix},\\
      \AA_{:, 4:5} &=
      \begin{pmatrix}
         c_4\beta^2 & 0 \\
         c_6\beta^2 & c_5\frac{\almax^2}{\beta}  \\
         \frac{\cerr{2}n c_7\beta^2}{\gamma} & c_8\gamma^2 + \frac{2c_7 \almax^2}{\gamma}  \\
         \AA_{44} & \frac{2\almax^2}{\eta}  \\
         0 & 0
      \end{pmatrix},
    }
    where  $\AA_{44} = 1 - \eta + 2\cerr{2}\eta^2 + 2n\cerr{2}\beta^2$ and $c_1$-$c_{10}$ are constants given by~(\ref{eq: def c 1}),~(\ref{eq: def c 2}),~(\ref{eq:defc3}) and (\ref{eq:defc4}) in the proof.
    Then, under Assumptions~\ref{assp: require each fi}-\ref{assp: compress} and the condition $\altil \leq \frac{2}{\mu + L}$, we have the following inequalities: 
    \begin{subequations}
        \small
      \begin{numcases}{}
        \expec{\mt{\YY^k}^2} \leq {\AA\pr{\almax, \beta, \gamma, \eta}}_{5, 1:4} \dd^{k}_{1:4}, \label{eq:cYmat1}  \\
        \dd^{k+1}_{1:4} \leq {\AA\pr{\almax, \beta, \gamma, \eta}}_{1:4, :} \dd^{k}.  \label{eq:cmat2}
      \end{numcases}
    \end{subequations}
\end{lemma}
\begin{proof}
    See Supplementary Material~\ref{sec:proofofCpushpullmatineq1}.
\end{proof}

Now, we proceed to show the spectral radius of $\AA\pr{\almax, \beta, \gamma, \eta}$ defined in (\ref{eq:defA1}) is less than $1$ with properly chosen parameters.
\begin{lemma}\label{lem:Arl1}
    Given any $\alp' > 0, \beta' > 0, w\geq \frac{n}{\vecr\tp\vecc}$ and $0 < \eta < \frac{1}{2\cerr{2}}$ with $\cerr{2} > 0$ given in Assumption~\ref{assp: compress},
    there exists $\gamma' > 0$ such that for any $0 < \gamma < \gamma'$, if we take $\almax = \alp'\gamma^3$ and $\beta = \beta'\gamma^2$, then
    \seq{ \rho\pr{\AA\pr{\almax, \beta, \gamma, \eta}} < 1.  }
    More specifically, we can choose $\gamma' = \sup_{\dr{\va_i}_{i=1}^5\in\calP}\min\dr{\prt_1, \prt_2, \prt_3, \prt_4}$, where $\calP = \dr{\dr{\va_i}_{i=1}^5 | \text{$\dr{\va_i}_{i=1}^5$ satisfies (\ref{eq:cvr1}), (\ref{eq:cvr2})}}$ and $\prt_1, \prt_2, \prt_3, \prt_4$ are the minimum positive roots of the polynomials (\ref{eq:defg1})-(\ref{eq:defg4}).
\end{lemma}
\begin{proof}
    To begin with, we choose positive numbers $\va_1, \va_2, \va_3, \va_4, \va_5$ satisfying
    \begin{subequations}
      \begin{numcases}{}
        c_2\va_2 + c_3\va_3 < c_1\va_1, \label{eq:cvr1}  \\
        c_{10}\va_1 + c_9\va_2 + 3\va_3 < \va_5 \label{eq:cvr2}  .
      \end{numcases}
    \end{subequations}
    Define $\va = \pr{\va_1, \va_2, \va_3, \va_4, \va_5}\tp, $
    then the five entries $\dr{g_i}_{i=1}^5$ of the vector $\br{\AA\pr{\alp'\gamma^3, \beta'\gamma^2, \gamma, \eta}\va - \va}  $ (as functions of $\gamma$) are given by
    \begin{subequations}
      \small
        \begin{align}
        &g_1\pr{\gamma} = \pr{- c_1\va_1 + c_2\va_2 + c_3\va_3}\alp'\gamma^3 \notag \\
        & \qquad\qquad  + c_4\va_4\beta'^2\gamma^4 \label{eq:defg1} \\
        &g_2\pr{\gamma} = -\theR\beta'\va_2\gamma^2 \notag \\
        &\qquad\qquad + \pr{c_6\va_4\beta'^2 + c_5 \va_5 \frac{\alp'^2}{\beta'}} \gamma^4  \label{eq:defg2} \\
        &g_3\pr{\gamma} = -\theC\va_3\gamma + c_8\va_5\gamma^2  \notag \\
        &\qquad\qquad + \pr{8  n c_7\va_2\beta'^2 + \cerr{2} n c_7 \va_4\beta'^2}\gamma^3  \notag \\
        &\qquad\qquad + 2c_7\va_5\alp'^2\gamma^5 \label{eq:defg3} \\
        &g_4\pr{\gamma} = - \pr{\eta - 2\cerr{2}\eta^2}\va_4 \notag \\
        &+ \pr{\frac{8  n\beta'^2\va_2}{\eta} + 2n\cerr{2}\beta'^2\va_4}\gamma^4 + \frac{2\alp'^2\va_5}{\eta}\gamma^6 \label{eq:defg4} \\
        &g_5\pr{\gamma} = c_{10}\va_1 + c_9\va_2 + 3\va_3 - \va_5  .
        \end{align}
    \end{subequations}
    By (\ref{eq:cvr1}), for any $0 < \gamma < \prt_1 = \frac{\alp'\pr{c_1\va_1 - c_2\va_2 - c_3\va_3}}{c_4\va_4\beta'^2} $,
    we have $g_1\pr{\gamma} < 0$,
    where $\prt_1$ is the minimum positive root of $g_1\pr{\gamma}$.

    By direct calculation, $g_2\pr{\gamma} < 0$, for any $0 < \gamma < \prt_2 = \sqrt{\frac{\theR\beta'\va_2}{c_6\va_4\beta'^2 + c_5 \va_5 \frac{\alp'^2}{\beta'}}} $,
    where $\prt_2$ is the minimum positive root of $g_2\pr{\gamma}$.

    Since $g_3\pr{0} = 0$, $g_3'\pr{0} = -\theC\va_3 < 0$ and $\lim_{\gamma\arr+\infty}g_3\pr{\gamma} = +\infty$, $g_3\pr{\gamma}$ has at least one positive root.
    By defining $\prt_3$ as this minimum positive root of $g_3\pr{\gamma}$, we have $g_3\pr{\gamma} < 0$, for any $0 < \gamma < \prt_3$.

    As we have assumed $0 < \eta < \frac{1}{2\cerr{2}}$ and $\va_4 > 0$, $g_4\pr{0} = - \pr{\eta - 2\cerr{2}\eta^2}\va_4 < 0$.
    Since $\lim_{\gamma\arr+\infty}  g_4\pr{\gamma} = +\infty$, $g_4\pr{\gamma}$ has a minimum positive root $\prt_4$.
    Then, we have $g_4\pr{\gamma} < 0$, for any $0 < \gamma < \prt_4$.
    From (\ref{eq:cvr2}), $g_5\pr{\gamma}$ is a negative constant.

    Thus, by the above discussion, for $0 < \gamma < \min$ $\dr{\prt_1, \prt_2, \prt_3, \prt_4}$,
    we have $g_i\pr{\gamma} < 0(1\leq i\leq 5)$, i.e., $ \AA\pr{\almax'\gamma^3, \beta'\gamma^2, \gamma, \eta}\va = \AA\pr{\almax, \beta, \gamma, \eta}\va < \va  . $
    Since $\va$ is a positive vector, $\AA\pr{\almax, \beta, \gamma, \eta}$ is a nonnegative matrix, by \cite[Corollary~8.1.29]{horn2012matrix}, $\rho\pr{\AA\pr{\almax, \beta, \gamma, \eta}} < 1.$
    The above arguments hold true for any positive numbers $\dr{\va_i}_{i=1}^5$ satisfying (\ref{eq:cvr1}) and (\ref{eq:cvr2}).
    Then, by defining $\calP$ as the set containing these positive number sets $\dr{\va_i}_{i=1}^5$, we can choose $\gamma'$ as $\sup_{\dr{\va_i}_{i=1}^5\in\calP}\min\dr{\prt_1, \prt_2, \prt_3, \prt_4}$.
\end{proof}

The next theorem shows that Algorithm~\ref{alg: push pull with compression} converges linearly with proper parameters and stepsizes.
\begin{theorem}\label{thm:linconv}
    Under Assumptions~\ref{assp: require each fi}-\ref{assp: compress},
    given positive numbers $\alp', \beta'$ and $0 < \eta < \frac{1}{2\cerr{2}}$, $\eta \leq 1$, $w \geq \frac{n}{\vecr\tp\vecc}$,
    there exists $\gamma' > 0$ (the value of $\gamma'$ is given in Lemma~\ref{lem:Arl1}) such that
    for any $\gamma$ satisfying
    $0 < \gamma < \gamma'$
    and
    \seql{\label{eq:ga'def2}}{ \gamma \leq \min\dr{1, \beta'^{-1/2}, \pr{\frac{2n}{\alp'\pr{\mu+L}\pr{\vecr\tp\vecc}}}^{1/3}}, }
    if we set $\beta = \beta'\gamma^2$
    and the stepsizes $\dr{\alp_i}_{i\in\calN}$ satisfy
    $\almax = \max_{i\in\calN}\alp_i = \alp'\gamma^3,  $
    with $\almax / \altil \leq w$,
    then
    $\expec{\nt{\xa^k - \xx^*}^2}$, $\expec{\mR{\PR\XX^k}^2}$ converge linearly to $0$.
    More specifically,
    \seq{
        \expec{\nt{\xa^k - \xx^*}^2} &\leq \RL \vr_1 \rho\pr{\AA\pr{\almax, \beta, \gamma, \eta}}^k,  \\
        \expec{\mR{\PR\XX^k}^2} &\leq \RL \vr_2 \rho\pr{\AA\pr{\almax, \beta, \gamma, \eta}}^k  ,
    }
    where $\RL, \vr_1, \vr_2$ are constants given in the proof.
\end{theorem}
\begin{proof}
    Denote $\AA = \AA\pr{\almax, \beta, \gamma, \eta}$ for simplicity.
    Since $\AA$ is a regular nonnegative matrix\footnote{A matrix $\AA$ is said to be regular if all the entries of $\AA^k$ are positive for some integer $k\geq 0$. }, by the Perron-Frobenius theorem~\cite{berman1994nonnegative}, $\rho\pr{\AA}$ is an eigenvalue of $\AA$, and
    $\AA$ has a unique right positive eigenvector $\vr$  with respect to the eigenvalue $\rho\pr{\AA}$.
    Define the constant
    $
        \RL = \max_{1\leq i\leq 4} \frac{\dd^0_i}{\vr_i}  .
    $
    Thus, by the definition of $\RL$,
    $
        \dd^{0}_{1:4} \leq \RL \vr_{1:4} .
    $

    Next, we prove the linear convergence by induction.
    If we have proved $\dd^{k}_{1:4} \leq \RL\rho\pr{\AA}^k\vr_{1:4}$, we will show that it also holds true for $k + 1$.
    The requirement (\ref{eq:ga'def2}) guarantees that $\gamma \leq 1$, $\beta = \beta'\gamma^2 \leq 1$, $\altil \leq  \frac{\pr{\vecr\tp\vecc}\almax}{n} = \frac{\pr{\vecr\tp\vecc}\alp'\gamma^3}{n} \leq \frac{2}{\mu + L} $.
    According to Lemma~\ref{lem:Arl1}, $\rho\pr{\AA} < 1$.

    By (\ref{eq:cYmat1}) and the inductive hypothesis, there holds
    \seq{
         &\expec{\mt{\YY^k}^2} \leq \AA_{5,1:4}\dd^{k}_{1:4} \leq \RL\rho\pr{\AA}^k\AA_{5,1:4}\vr_{1:4}  \\
        =& \RL\rho\pr{\AA}^k\AA_{5,:}\vr
        = \RL\rho\pr{\AA}^{k+1}\vr_5 \leq \RL\rho\pr{\AA}^k\vr_5  .
    }
    Combining with the inductive hypothesis, we obtain
    \seql{\label{eq:dkRLrhokvr1}}{
        \dd^{k} \leq \RL\rho\pr{\AA}^k\vr  .
    }
    Now, by (\ref{eq:cmat2}) and (\ref{eq:dkRLrhokvr1}),
    \seq{
        \dd^{k+1}_{1:4} \leq \AA_{1:4, :}\dd^{k} \leq \RL\rho\pr{\AA}^k\AA_{1:4,:}\vr
        = \RL\rho\pr{\AA}^{k+1}\vr_{1:4}  ,
    }
    i.e., the statement holds for $k + 1$.
    Therefore, by induction,
    $
        \dd^{k}_{1:4} \leq \RL\rho\pr{\AA}^{k}\vr^{k}_{1:4}  ,
    $
    for any $k \geq 0$, which completes the proof.
\end{proof}

\begin{remark}\label{rem:chooseparastp1}
    In practice, we can take for instance $\alp' = \frac{1}{L}$, $\beta' = 1$, $\gamma = 1$,  $\eta = \min\dr{\frac{1}{2\cerr{2}}, 1}$.
    Then, we hand-tune $\gamma$ under the relations $\alp_i = \almax = \alp'\gamma^3 (\forall i\in\calN)$ and $\beta = \beta'\gamma^2$ to achieve the optimal  performance.
\end{remark}
\begin{remark}
    We can also add momentums for the communication of $\YY^k$ like what we did for $\XX^k$.
    In this case, linear convergence can be proved similarly.
    However, we see little improvement in the numerical experiments when $\YY^k$ is equipped with momentums.
    In addition, more momentums will take more storage space.
    Therefore, we omit the details of this case here.
\end{remark}

\section{A Broadcast-like Gradient Tracking Method with Compression (B-CPP)}
\label{sec:BCPP}
In this section, we consider a broadcast-like gradient tracking method  with compression (B-CPP). Broadcast or gossip-based communication protocols are popular for distributed computation due to their low communication costs \cite{aysal2009broadcast,boyd2006randomized,pu2020distributed}.
In B-CPP, at every iteration $k$, one agent $i_k\in \calN$ wakes up with uniform probability. This can be easily realized, for example, if each agent wakes up according to an independent Poisson clock with the same parameter. Hence the probability $\prob{i_k = j} = \frac{1}{n}$ for any $j\in \calN$.
In addition, $\dr{i_k}_{k\geq 0}$ are independent with each other.

Briefly speaking, each iteration $k$ of the B-CPP algorithm consists of the following steps:
\begin{enumerate}
  \item One random agent $i_k$ wakes up.
  \item Agent $i_k$ sends information to all its out-neighbors in $\RR$ ($\calN_{\rmR, i_k}^+$) and  $\CC$ ($\calN_{{\rm C}, i_k}^+$).
  \item The agents who received information from $i_k$ are awaken and update their local variables.
\end{enumerate}
We remark that the number $k$ is only used for analysis purpose and does not need to be known by the agents.
B-CPP can be generalized to the case when each agent wakes up with different {but known probabilities}.
The awakened agents will know whether they are $i_k$ and whether they are in the set $\calN_{\rmR, i_k}^+$ or in the set $\calN_{\rmC, i_k}^+$, and will take different actions accordingly.
The B-CPP method is illustrated in Algorithm~\ref{alg: BC-Push-Push}.
\begin{algorithm}
\caption{B-CPP }
\label{alg: BC-Push-Push}
\KwIn{each agent $i$ is instilled its initial position $\xx_i$, parameters $\beta, \gamma, \eta \in (0, 1]$, stepsize $\alp_i$ and network information $\RR_{ij}(j\in \calN_{\rmR, i}^-)$, $\CC_{ij}(j\in \calN_{\rmC, i}^-)$ and the sets $\calN_{\rmR, i}^+$, $\calN_{\rmC, i}^+$
}
\KwOut{each agent $i$ outputs its final local decision variable $\xx_i$
}

Each agent $i$ initializes $\yy_i^0 = \na f_i\pr{\xx_i^0}, \uu_i^{0} = \zero, \uu_{\rmR, i}^{0} = \zero$. \;
\For{ $k=0$ to $K-1$  }{
    Agent $i_k$ awakes. \;
    Agent $i_k$ compresses $\pp^k = \compress\pr{\xx_{i_k} - \uu_{i_k}}$, sends $\pp^k$ to all agents $j\in \calN_{\rmR, i_k}^+$ and wakes up these agents. \;
    Agent $i_k$ compresses $\qq^k = \compress\pr{\yy_{i_k}}$, sends $\qq^k$ to all agents in $\calN_{\rmC, i_k}^+$ and wakes up these agents. \;
    For all agents $j\in \dr{i_k}\cup\calN_{\rmR, i_k}^+\cup\calN_{\rmC, i_k}^+$, agent $j$ records $\gratemp_j \arrlf \nabla f_j\pr{\xx_{j}} $.  \;
    For $j\in \calN_{\rmR, i_k}^+$, agent $j$ updates
    \begin{subequations}\label{eq:bxupdate}
    \small
    \begin{align}
         &\uu_{\rmR, j} \arrlf \uu_{\rmR, j} + \eta n \RR_{j, i_k}\pp^k , \label{eq:bxupdate1}  \\
         &\xx_j \arrlf \pr{1 - \frac{\beta n}{\degR_j}}\xx_j \notag \\
         & \qquad\quad + \frac{\beta n}{\degR_j}\uu_{\rmR, j} + \beta n\RR_{j, i_k}\pp^k  . \label{eq:bxupdate2}
    \end{align}
    \end{subequations}
    Agent $i_k$ updates $\uu_{i_k} \arrlf \uu_{i_k} + \eta n \pp^k $. \;
    For all $j\in \dr{i_k}\cup\calN_{\rmR, i_k}^+\cup\calN_{\rmC, i_k}^+$, agent $j$ updates $\xx_j \arrlf \xx_j - \alp_j\yy_j$ . \;
    For all $j\in \dr{i_k}\cup\calN_{\rmR, i_k}^+ \cup \calN_{\rmC, i_k}^+$, agent $j$ updates
        $\yy_j \arrlf \yy_j + \nabla f_j\pr{\xx_j} - \gratemp_j  $. \;
    Agent $i_k$ update $\yy_{i_k} \arrlf \yy_{i_k} - \gamma n \qq^k $. \;
    For $j\in \calN_{\rmC, i_k}^+$, agent $j$ updates
        $\yy_j \arrlf \yy_j + \gamma n \CC_{ji_k} \qq^k .   $
}
Return the final local decision variable $\dr{\xx_i}_{i\in\calN}$.

\end{algorithm}
%
%
%
To implement communication compression in a broadcast setting, a naive way is to let each agent $i$ hold different momentums $\uu_{i, j}$ for each neighbors $j$.
In this way, at each time agent $i$ receives information $\pp$ from neighboring agent $j$, it can restore the information sent from agent $j$ directly by summing $\pp + \uu_{i,j}$.
However, this procedure would require momentums as many as the total number of edges. Hence it could be impractical when the storage space is limited and the graph is dense.
In B-CPP, we overcome this issue so that each agent only uses 2 momentums.

To help analyze the convergence rate of Algorithm~\ref{alg: BC-Push-Push},
for $0\leq k\leq K$, let us define the final state of variable $\xx_j$ before iteration $k$ as $\xx_j^k$.
These $\xx_j^k$ are written compactly into an $n$-by-$p$ matrix $\XX^k$. And $\YY^k, \UU^k, \UU_{\rmR}^k, \gF{k}$ are defined analogously.
Let $\calD_k$ be the $\sigma$-field generated by $\dr{i_j, \pp^j, \qq^j}_{0\leq j\leq k-1 }$.
Let $\calD_k^+$ denote the $\sigma$-field generated by $\calD_k$ and $i_k$.
Let $\calD_k^{++}$ denote the $\sigma$-field generated by $\calD_k$ and $\dr{i_k, \pp^k}$.
For any event $A$, let $\chi_A$ denote the indicator function of $A$.

{
The differences between B-CPP and CPP mainly lie in the averaging step.
Taking the averaging step for $\XX^k$ as an example.
Firstly,
we also have by induction that
$
  \UU_{\rmR}^k = \RR\UU^k,\ \ \forall 0\leq k\leq K  .
$}
For $i\in\calN$, denote the in-degree of $i$ in $\RR$ by $\degR_i = \abs{\calN_{\rmR, i}^-} $.
Since $\RR$ is row stochastic, $r_i > 0$ for any $i\in \calN$.
Notice that the update step (\ref{eq:bxupdate1}) will be implemented by agent $j$ when agent $i_k$ is an in-neighbor of agent $j$ (or equivalently, $j$ is an out-neighbor of $i_k$).
\seq{
    &\expec{\chi_{\dr{j\in \calN_{\rmR, i_k}^+}} n\RR_{j, i_k}\pp^k | \calD_k}  \\
    =& \expec{\chi_{\dr{j\in \calN_{\rmR, i_k}^+}} n\RR_{j, i_k}\expec{\pp^k | i_k} | \calD_k}  \\
    =& \expec{ n\RR_{j, i_k}\pr{\xx_{i_k}^k - \uu_{i_k}^k} | \calD_k}\\
    =& \frac{1}{n}\sum_{i_k=1}^n n\RR_{j, i_k}\pr{\xx_{i_k}^k - \uu_{i_k}^k} \\
    =& \RR_{j,:}\pr{\XX^k - \UU^k}  .  
}
Analogously,
\seq{
    &\expec{\chi_{\dr{j\in \calN_{\rmR, i_k}^+}} \pr{1 - \frac{\beta n}{\degR_j}}\xx_j^k  | \calD_k}  \\
    =& \frac{1}{n}\sum_{i_k\in\calN,\ \RR_{j,i_k} > 0}\pr{1 - \frac{\beta n}{\degR_j}}\xx_j^k  \\
    =& \frac{r_j}{n}\xx_j^k - \beta \xx_j^k  ,
}
and
\seq{
    &\expec{\chi_{\dr{j\in \calN_{\rmR, i_k}^+}}\pr{\frac{n}{\degR_j}\uu_{\rmR, j}^k} | \calD_k}  \\
    =&  \frac{n}{\degR_j}\uu_{\rmR, j}^k \prob{j\in \calN_{\rmR, i_k}^+ | \calD_k  }  \\
    =& \frac{1}{n}\sum_{i_k\in\calN,\ \RR_{j, i_k} > 0}\frac{n}{\degR_j} \uu_{\rmR, j}^k \\  
    =& \uu_{\rmR, j}^k
    = \RR_{j,:}\UU^k  .
}
Thus, taking expectation on the RHS of (\ref{eq:bxupdate2}) yields
\begin{equation*}
    \small
    \begin{split}
      &\mathbb{E}\Big[\chi_{\dr{j\notin \calN_{\rmR, i_k}^+}} \xx_j^k+  \chi_{\dr{j\in \calN_{\rmR, i_k}^+}}\pr{\pr{1 - \frac{\beta n}{\degR_j}}\xx_j^k}  \\
    & +\chi_{\dr{j\in \calN_{\rmR, i_k}^+}}\pr{\frac{\beta n}{\degR_j}\uu_{\rmR, j}^k + \beta n\RR_{j, i_k}\pp^k} \Big| \calD_k\Big]  \\
    =& \left(\frac{n - \degR_j}{n} + \frac{r_j}{n} - \beta\right) \xx_j^k + \beta\RR_{j,:}\pr{\UU^k+\XX^k - \UU^k}  \\
    =& \br{\Rb}_{j,:}\XX^k  .
    \end{split}
\end{equation*}
{
    Briefly speaking, after taking conditional expectation, the averaging step of B-CPP reduces to that of CPP (the first term of the RHS of (\ref{eq: X update k+1})).
By choosing a proper value for $\beta$, the variance of the RHS of (\ref{eq:bxupdate}) can be made small enough.
In this way, the averaging step in B-CPP could have a similar effect to that of CPP.
}

To show the linear convergence of B-CPP,
for simplicity, we assume $\alp = \alp_1 = \alp_2 = \cdots = \alp_n$.
We remark that the analysis below can be easily generalized to the case when the stepsizes differ among the agents.


We denote $\one_i\in\Real^{n\times 1}$ as the vector with $1$ on the $i$-th component and $0$ on the others.
Define
\seq{
     \PP^k &= \XX^k - \UU^k + \one_{i_k}\pr{\pp^k - \xx_{i_k}^k + \uu_{i_k}^k}, \\ \QQ^k &= \YY^k + \one_{i_k}\pr{\qq^k - \yy_{i_k}^k}  .
}
As $\UU^k + \PP^k$ and $\QQ^k$ are used to estimate $\XX^k$ and $\YY^k$, respectively, we define the error matrices induced by the compression as follows
\seq{
    \EX^k &= \XX^k - \UU^k - \PP^k = \one_{i_k}\pr{\xx_{i_k}^k - \uu_{i_k}^k - \pp^k}  ,  \\
    \EY^k &= \YY^k - \QQ^k = \one_{i_k}\pr{\yy_{i_k}^k - \qq^k}  .
}
The following lemma is a direct corollary of Assumption~\ref{assp: compress}.
\begin{lemma}\label{assp:bindept.}
    The random variable $i_k$ is independent with $\calD_k$.
    And
    $
        \expec{\EX^k | \calD_k^+} = \zero,\
        \expec{\EY^k | \calD_k^{++}} = \zero  .
    $
\end{lemma}
Moreover, by Assumption~\ref{assp: compress},
{\small
\begin{align}
  &\expec{\mt{\EX^k}^2 | \calD_{k}^+} = \expec{\nt{\xx_{i_k}^k - \uu_{i_k}^k - \pp^k}^2|\calD_{k}^+}  \notag    \\
  \leq& \cerr{2}\expec{\nt{\xx_{i_k}^k - \uu_{i_k}^k}^2 | \calD_k} = \frac{\cerr{2}}{n} \mt{\XX^k - \UU^k}^2 , \label{eq:bEX}  \\
  &\expec{\mt{\EY^k}^2 | \calD_{k}^{++}} = \expec{\nt{\yy_{i_k} - \qq^k}^2|\calD_{k}} \notag    \\
  \leq& \cerr{2}\expec{\nt{\yy_{i_k}^k}^2 | \calD_k^{++}} = \frac{\cerr{2}}{n} \mt{\YY^k}^2 . \label{eq:bEY}
\end{align}
}
We will also use the following random matrices, which are all measurable with respect to $i_k$,
\seq{
  \IW^k &= n\Diag{\one_{i_k}},
  \AW^k = \sum_{j\in \dr{i_k}\cup\calN_{\rmR, i_k}^+\cup\calN_{\rmC, i_k}^+} \Diag{\one_j}, \\
  \RW^k &= \sum_{j\in \calN_{\rmR, i_k}^+} \frac{n}{\degR_j}\Diag{\one_j},
  \RU^k = \sum_{j\in \calN_{\rmR, i_k}^+} \frac{n}{\degR_j}\one_j\RR_{j, :},  \\
  \RP^k &= n\RR_{:, i_k}\one_{i_k}\tp,
  \CQ^k = n\CC_{:, i_k}\one_{i_k}\tp  .
}
Now, Algorithm~\ref{alg: BC-Push-Push} can be rewritten into a more compact form:
\begin{subequations}\label{eq: B-push-pull}
  \small
  \begin{numcases}{}
    \XX^{k+1} = \pr{\II - \beta\RW^k} \XX^k + \beta\RW^k \RR\UU^k \notag \\
    \qquad\qquad + \beta\RP^k \PP^k
     - \alp \AW^k \YY^k, \label{eq:bXupdate}  \\
    \YY^{k+1} = \YY^k + \gamma\pr{\CQ^k - \IW^k}\QQ^k \notag \\
    \qquad\qquad + \gF{k+1} - \gF{k},  \label{eq:bYupdate}  \\
    \UU^{k+1} = \UU^k + \eta\IW^k\PP^k.  \label{eq:bUupdate}
  \end{numcases}
\end{subequations}
{
Compared to CPP, the stochastic matrices $\RP^k, \RU^k, \IW^k, \AW^k, \RW^k$ will induce additional errors that need to be dealt with in the convergence analysis of B-CPP.
}

To analyze the convergence rate of B-CPP, similar to what we did for CPP, we also use $\expec{\nt{\xa^k - \xx^*}^2}$ ,$\expec{\mR{\PR\XX^k}^2}, \expec{\mC{\PC\YY^k}^2}$ and $\expec{\mt{\UU^k - \XX^k}^2}$ to measure the closeness to the optimal point, consensus error, gradient tracking error and the convergence of the momentums, respectively.
And we  use $\expec{\mt{\YY^k}^2}$, $\expec{\mt{\XX^{k+1}- \XX^k}^2}$ and $\expec{\mt{\ro^k}^2}$ to help simplify the recursive relations between the above four quantities.
The definition of $\ro^k$ is given in (\ref{defro1}).

\begin{lemma}\label{lem:bmatineq1}
    For $k \geq 0$, define $\db^k \in \Real^7$ as an extension of $\dd^k$:
    \seq{
        \db^{k} =
        \Big(
           \dd^k, \expec{\mt{\ro^k}^2}, \expec{\mt{\XX^{k+1} - \XX^k}^2}
        \Big)\tp  .
    }
    Define
    a parameterized matrix $\Ab\pr{\alp, \beta, \gamma, \eta}$ as follows
    \seql{\label{eq:bdefA}}{
         \Ab_{:, 1:3} &=
        \begin{pmatrix}
          1 - d_1\alp & d_2\alp & d_3\alp   \\
          0 & 1 - \theR\beta & 0    \\
          0 & 0 & 1 - \theC\gamma    \\
          0 & 0 & 0   \\
           c_{10} & c_9 & 3   \\
          0 & 3\Vara\beta^2 & 0    \\
          0 & 18\beta^2 & 0
        \end{pmatrix}  , \\
        \Ab_{:, 4:7} &=
        \begin{pmatrix}
           0 & 0 & d_4 & 0  \\
           0 & \frac{\alp^2}{\beta\theR} & \traRt^2 & 0   \\
           0 & \frac{d_5\gamma^2}{1 - \theC\gamma} + d_7\gamma^2 & 0 & \frac{d_6}{\gamma}   \\
           \Ab_{44} & 0 & 0 & \frac{2}{\eta}  \\
            0 & 0 & 0 & 0  \\
           d_{9}\beta^2 & 3\Vard\alp^2 & 0 & 0   \\
           0 & 2\alp^2 & 1 & 0
        \end{pmatrix}
    }
    where $\Ab_{44} = 1 - \eta + \frac{\eta^2(n-1)^2}{1 - \eta} + d_8\frac{\beta^2}{\eta} + 2\cerr{2}n\eta^2 + d_8\beta^2$ and
    the constants $d_1$-$d_{9}$ are given by~(\ref{defd1}),~(\ref{defd2}),~(\ref{eq:defd3}) and~(\ref{eq:defd4}) in the proof.
    Then, under Assumptions~\ref{assp: require each fi}-\ref{assp: compress} and the condition $\alb \leq \frac{2}{\mu + L}$, we have the following inequalities: 
    \begin{subequations}
        \small
      \begin{numcases}{}
        \expec{\mt{\YY^k}^2} \leq \Ab\pr{\alp, \beta, \gamma, \eta}_{5,1:4}\db^{k}_{1:4} , \label{eq:bYkAb1}  \\
        \expec{\mt{\ro^k}^2} \leq \Ab\pr{\alp, \beta, \gamma, \eta}_{6,1:5}\db^{k}_{1:5} , \label{eq:brokAb1}  \\
        \expec{\mt{\XX^{k+1} - \XX^k}^2} \leq \Ab\pr{\alp, \beta, \gamma, \eta}_{7,1:6}\db^{k}_{1:6} ,  \label{eq:bXk+1kAb1}  \\
        \db^{k+1}_{1:4} \leq \Ab\pr{\alp, \beta, \gamma, \eta}_{1:4, :}\db^k  .  \label{eq:bdk+1Ab1}
      \end{numcases}
    \end{subequations}
\end{lemma}
\begin{proof}
    See Supplementary Material~\ref{sec:pfbmatineq1}.
\end{proof}

Next, we show that with properly chosen parameters and stepsize, the spectral radius of the parameterized matrix $\Ab\pr{\alp, \beta, \gamma, \eta}$ in (\ref{eq:bdefA}) will be less than~$1$.
\begin{lemma}\label{lem:bxpa1}
    Given any positive numbers $\alp', \beta'$ and $\eta$ satisfying
    \seql{\label{eq:bpaxeta1}}{0 < \eta \leq \frac{1}{2}, \quad \eta < \frac{1}{2(n-1)^2 + 2\cerr{2}n}, }
    there exists $\gamma' > 0$ such that for any $0 < \gamma < \gamma'$,
    if we take $\alp = \alp'\gamma^3$ and $\beta = \beta'\gamma^2$, then
    \seq{ \rho\pr{\Ab\pr{\alp, \beta, \gamma, \eta}} < 1  . }
    More specifically, we can take $\gamma' = \sup_{\dr{\vo_i}_{i=1}^7\in\calP'}\min\dr{\pt_2, \pt_3, \pt_4, \pt_6, \pt_7, \frac{1}{\theC}}$, where $\theC$ is given in Lemma~\ref{def nR nC} and $\calP', \pt_i(i=2,3,4,6,7)$ are given in the proof.
\end{lemma}
\begin{proof}
    Given positive numbers $\dr{\vo_i}_{i=1}^7$ satisfying
    \begin{subequations}
      \small
      \begin{numcases}{}
        d_2\vo_2\alp' + d_3\vo_3\alp' + d_4\vo_6 < d_1\vo_1\alp' , \label{eq:bpax1}  \\
        c_{10}\vo_1 + c_9\vo_2 + 3\vo_3 < \vo_5, \label{eq:bpax2} \\
        \vo_6 < \vo_7  \label{eq:bpax3} .
      \end{numcases}
    \end{subequations}
    Define
    $
        \ve = \pr{\vo_1, \vo_2, \vo_3, \vo_4, \vo_5, \gamma^3\vo_6, \gamma^3\vo_7  }\tp  .
    $
    The seven entries $\dr{g_i}_{i=1}^7$ of the vector $\br{\ve - \Ab\pr{\alp'\gamma^3, \beta'\gamma^2, \gamma, \eta}\ve}$ (as functions of $\gamma$) are given by
    \begin{equation*}
      \small
      \begin{split}
        g_1\pr{\gamma} &= \gamma^3 \pr{- d_1\vo_1\alp' + d_2\vo_2\alp' + d_3\vo_3\alp' + d_4\vo_6}   \\
        g_2\pr{\gamma} &= \gamma^2 \pr{ -\theR\beta'\vo_2 + \frac{\alp'^2\vo_5}{\beta'\theR}\gamma^2 + \traRt^2\vo_6\gamma},  \\
        g_3\pr{\gamma} &= \gamma \pr{-\theC\vo_3 + \frac{d_5\vo_5\gamma}{1 - \theC\gamma} + d_7\vo_5\gamma + d_6\vo_7\gamma}, \\
        g_4\pr{\gamma} &= \pr{- \eta + \frac{\eta^2(n-1)^2}{1 - \eta} + 2\cerr{2}n\eta^2}\vo_4 + d_8\vo_4\frac{\beta'^2}{\eta}\gamma^4  \notag  \\
        &\qquad + d_8\vo_4\beta'^2\gamma^4 + \frac{2\vo_7}{\eta}\gamma^3,  \\
        g_5\pr{\gamma} &= c_{10}\vo_1 + c_9\vo_2 + 3\vo_3 - \vo_5,  \\
        g_6\pr{\gamma} &= \gamma^3 \Big( 3\Vara\vo_2\beta'^2\gamma + d_{9}\vo_4\beta'^2\gamma
          + 3\Vard\vo_5\alp'^2\gamma^3-\vo_6 \Big), \\
        g_7\pr{\gamma} &= \gamma^3 \pr{ \pr{\vo_6 - \vo_7} + 18\vo_2\beta'^2\gamma + 2\vo_5\alp'^2\gamma^3 }.
      \end{split}
    \end{equation*}
    Denote $\gtil_i = g_i/\gamma^3(i = 1, 6, 7)$, $\gtil_2 = g_2/\gamma^2$, $\gtil_3 = g_3/\gamma$.
    By (\ref{eq:bpax1}), $\gtil_1 = - d_1\vo_1\alp' + d_2\vo_2\alp' + d_3\vo_3\alp' + d_4\vo_6 < 0. $
    By (\ref{eq:bpax2}), $g_5 = c_{10}\vo_1 + c_9\vo_2 + 3\vo_3 - \vo_5 < 0. $
    By (\ref{eq:bpax3}), $\gtil_7(0) = \vo_6 - \vo_7 < 0 .  $
    By (\ref{eq:bpaxeta1}), $\eta \leq 1/2$, then $g_4(0) \leq - \eta + 2\eta^2(n-1)^2 + 2\cerr{2}n\eta^2 < 0 .  $
    It follows directly that $\gtil_2(0) = -\theR\beta'\vo_2 < 0, $
    $\gtil_3(0) = -\theC\vo_3 < 0,  $
    $\gtil_6(0) = -\vo_6 < 0 . $

    If $0 < \gamma < \frac{1}{\theC}$, we have $\lim_{\gamma\arr+\infty}g_3(\gamma) = +\infty$.
    Combining with the facts that $\lim_{\gamma\arr+\infty}g_i(\gamma) = +\infty$, for $i=2,4,6,7$,
    we can define $\pt_i$ to be the minimum positive solution of $g_i(\gamma)$ $(i=2,3,4,6,7)$ when $0 < \gamma < \frac{1}{\theC}$.

    Thus, for any $0 < \gamma < \min\dr{\pt_2, \pt_3, \pt_4, \pt_6, \pt_7, \frac{1}{\theC}},$
    there holds $g_i(\gamma) < 0,$
    for any $1\leq i\leq 7,$
    i.e.,
    $\Ab\pr{\alp, \beta, \gamma, \eta}\ve = \Ab\pr{\alp'\gamma^3, \beta'\gamma^2, \gamma, \eta}\ve < \ve . $
    Since $\Ab\pr{\alp, \beta, \gamma, \eta}$ is a nonnegative matrix, $\ve$ is a positive vector,
    by \cite[Corollary~8.1.29]{horn2012matrix}, we have $\rho\pr{\Ab\pr{\alp, \beta, \gamma, \eta}} < 1. $

    The above arguments hold true for arbitrary $\dr{\vo_i}_{i=1}^7$ satisfying (\ref{eq:bpax1}), (\ref{eq:bpax2}) and (\ref{eq:bpax3}).
    Then, let the set $\calP'$ consist of all these $\dr{\vo_i}_{i=1}^7$, we have
    $\gamma' \geq \sup_{\dr{\vo_i}_{i=1}^7\in\calP'}\min\dr{\pt_2, \pt_3, \pt_4, \pt_6, \pt_7, \frac{1}{\theC}}. $
%
%
%
\end{proof}

The following theorem shows the linear convergence of Algorithm~\ref{alg: BC-Push-Push} given proper parameters and stepsize.
\begin{theorem}\label{thm:blc}
    Under Assumptions~\ref{assp: require each fi}-\ref{assp: compress},
    if we choose arbitrary positive numbers $\alp', \beta', \eta$
    satisfying $0 < \eta \leq \frac{1}{2}$, $\eta < \frac{1}{2(n-1)^2 + 2\cerr{2}n}, $
    and choose $\gamma' > 0$ as in Lemma~\ref{lem:bxpa1}.
    Then, for any $\gamma$ satisfying $0 < \gamma < \gamma'$ and
    \seql{\label{eq:bgaar1}}{ \gamma \leq \min\dr{1, \beta'^{-{1\over2}}, \pr{\frac{2n}{\alp'\pr{\mu+L}\pr{\vecr\tp\AW\vecc}}}^{1\over3}},  }
    $\beta = \beta'\gamma^2$, and $\alp = \alp'\gamma^3$,
    there holds, for any $k \geq 0$,
    \seq{
        \expec{\nt{\xa^k - \xx^*}^2} &\leq \cb\vb_1 \rho\pr{\Ab\pr{\alp, \beta, \gamma, \eta}}^k,  \\
        \expec{\mR{\PR\XX^k}^2} &\leq \cb\vb_2 \rho\pr{\Ab\pr{\alp, \beta, \gamma, \eta}}^k  ,
    }
    where $\cb, \vb_1, \vb_2$ are constants given in the proof.
\end{theorem}
\begin{proof}
    To begin with,
    we denote $\Ab = \Ab\pr{\alp, \beta, \gamma, \eta}$ for simplicity.
    Since $\Ab$ is a regular nonnegative matrix, by the Perron-Frobenius theorem~\cite{berman1994nonnegative}, $\rho\pr{\Ab}$ is an eigenvalue of $\Ab$, and $\Ab$ has a unique positive right eigenvector $\vb$ with respect to the eigenvalue $\rho\pr{\Ab}$.
    Define
    $
        \cb =
        \max_{1\leq i\leq 4} \frac{\db^0_i}{\vb_i} .
    $
    Then, $\db^{0}_{1:4} \leq \cb\vb_{1:4} . $

    Next, we prove the linear convergence by induction.
    If we have proved $\db^{k}_{1:4} \leq \cb\rho\pr{\Ab}^k\vb_{1:4} $, we will show that it also holds for $k+1$.
    The requirement (\ref{eq:bgaar1}) guarantees that $\gamma \leq 1, $ $\beta \leq 1, $ $\alb \leq \frac{2}{\mu + L}. $
    From Lemma~\ref{lem:bxpa1}, $\rho\pr{\Ab} < 1.$

    By (\ref{eq:bYkAb1}) and the inductive hypothesis, there holds
    \seq{&\expec{\mt{\YY^k}^2} \leq \Ab_{5,1:4}\db^{k}_{1:4} \leq \cb\rho\pr{\Ab}^k\Ab_{5,1:4}\vb_{1:4} \\
    =&
    \cb\rho\pr{\Ab}^k\br{\Ab\vb}_5 =
    \cb\rho\pr{\Ab}^{k+1}\vb_5 \leq \cb\rho\pr{\Ab}^k\vb_5  .    }
    Then combining with the inductive hypothesis, we have
    \seql{\label{eq:bdb51}}{\db^{k}_{1:5} \leq \cb\rho\pr{\Ab}^k\vb_{1:5} . }

    By (\ref{eq:brokAb1}) and~\eqref{eq:bdb51},
    \seq{&\expec{\mt{\ro^k}^2} \leq \Ab_{6,1:5}\db^{k}_{1:5} \leq \cb\rho\pr{\Ab}^k\Ab_{6,1:5}\vb_{1:5} \\
    =& \cb\rho\pr{\Ab}^k\br{\Ab\vb}_6 = \cb\rho\pr{\Ab}^{k+1}\vb_6 \leq \cb\rho\pr{\Ab}^k\vb_6  .    }
    Then combining with (\ref{eq:bdb51}), we have
    \seql{\label{eq:bdb61}}{\db^{k}_{1:6} \leq \cb\rho\pr{\Ab}^k\vb_{1:6} . }

    By (\ref{eq:bXk+1kAb1}) and~\eqref{eq:bdb61},
    \seq{&\expec{\mt{\XX^{k+1} - \XX^k}^2} \leq \Ab_{7,1:6}\db^{k}_{1:6} \leq \cb\rho\pr{\Ab}^k\Ab_{7,1:6}\vb_{1:6} \\
     =& \cb\rho\pr{\Ab}^k\br{\Ab\vb}_7 = \cb\rho\pr{\Ab}^{k+1}\vb_7 \leq \cb\rho\pr{\Ab}^k\vb_7  .    }
    Then combining with (\ref{eq:bdb61}), we have obtained
    \seql{\label{eq:bdbfull1}}{\db^{k} \leq \cb\rho\pr{\Ab}^k\vb . }

    Using (\ref{eq:bdk+1Ab1}) and (\ref{eq:bdbfull1}),
    we have
    \seq{
        &\db^{k+1}_{1:4} \leq \Ab_{1:4, :}\db^k \leq \cb\rho\pr{\Ab}^k \Ab_{1:4, :}\vb \\
        =& \cb\rho\pr{\Ab}^k \br{\Ab\vb}_{1:4} = \cb\rho\pr{\Ab}^{k+1}\vb_{1:4}  ,
    }
    i.e., the statement holds for $k + 1$.
    Therefore, by induction,
    $
        \db^{k}_{1:4} \leq \cb\rho\pr{\Ab}^k\vb_{1:4} ,
    $
    for any $k \geq 0$, which completes the proof.
\end{proof}

\begin{remark}
    In practice, the parameters $\beta, \gamma, \eta$ and stepsize $\alp$ can be chosen in the same way as in CPP.
\end{remark}

\section{Numerical Experiments}
\label{sec:numerical}
{
In this section, we compare the numerical performance of CPP and B-CPP with the Push-Pull/$\calA\calB$ method~\cite{pu2020push, xin2018linear}.
In the experiments, we equip CPP and B-CPP with different compression operators and consider different graph topologies.

We consider the following decentralized $\ell_2$-regularized logistic regression problem:
\seq{
   \min_{\xx\in \Real^p}  \frac{1}{n} \sum_{i=1}^{n}  \log\pr{1 + e^{-\la_{i}\zz_{i}\tp\xx}} + \frac{\mu}{2}\nt{\xx}^2 ,
}
where $\pr{\zz_{i}, \la_{i}}\in \Real^{p}\times\dr{-1, +1}$ is the  training example possessed by agent $i$.
The data is from QSAR biodegradation Data Set\footnote{\url{https://archive.ics.uci.edu/ml/datasets/QSAR+biodegradation}} ~\cite{mansouri2013quantitative}, where each feature vector is of $p = 41$ dimension.
In our experiments, we set $n = 20$ and $\mu = 0.001$.
We construct the directed graphs $\calG_{\RR}$ and $\calG_{\CC} $ by adding $d$ directed links to the undirected cycle, respectively.
In our experiments below, we will consider the cases $d = 5, 20, 50$, representing sparse, normal and dense graphs respectively.

We consider two kinds of compression operators.
The first one is the $b$ bit 2-quantization given by
\seq{
  \compress\pr{\xx} = \pr{\nt{\xx}\sign{\xx}2^{1-b  }} \floor{\frac{2^{b-1}\abs{\xx}}{\nt{\xx}} + \vv} ,
}
where $\vv \in \Real^p$ is a random vector picked uniformly from $(0,1)^p$ and we choose $b = 2, 4, 6$ in the experiments.
The second one is Rand-k which is defined by
\eq{
    \compress\pr{\xx} = \frac{p}{k} \widehat{\xx},
}
where $\widehat{\xx}$ is obtained by randomly choosing $k$ entries of $\xx$ and setting the rest entries to be $0$.
We will set $k = 5, 10, 20$ respectively.
Then,
we hand-tune to find the optimal parameters for all the numerical cases below.
In the figures, the $y$-axis represents the loss $f\pr{\xa^i} - f\pr{\xx^*}$ and $x$-axis is the iteration number or the number of transmitted bits.


\subsection{Linear convergence}
\label{subsec: numeric_linear}
The performance of Push-Pull/$\calA\calB$, CPP and B-CPP is illustrated in Fig.~\ref{fig:fig1}.
To see the performance of B-CPP more clearly, we plot the trajectories of B-CPP additionally at a larger scale in Fig.~\ref{fig:BCPP}.
The experiments illustrated in Fig.~\ref{fig:fig1} and Fig.~\ref{fig:BCPP} are conducted on graphs with $d = 20$.
The quantization with $b = 2, 4, 6$ and Rand-k with $k = 5, 10, 20$ are considered.
\begin{figure}[!ht]
    \centering
    \includegraphics[scale=0.39]{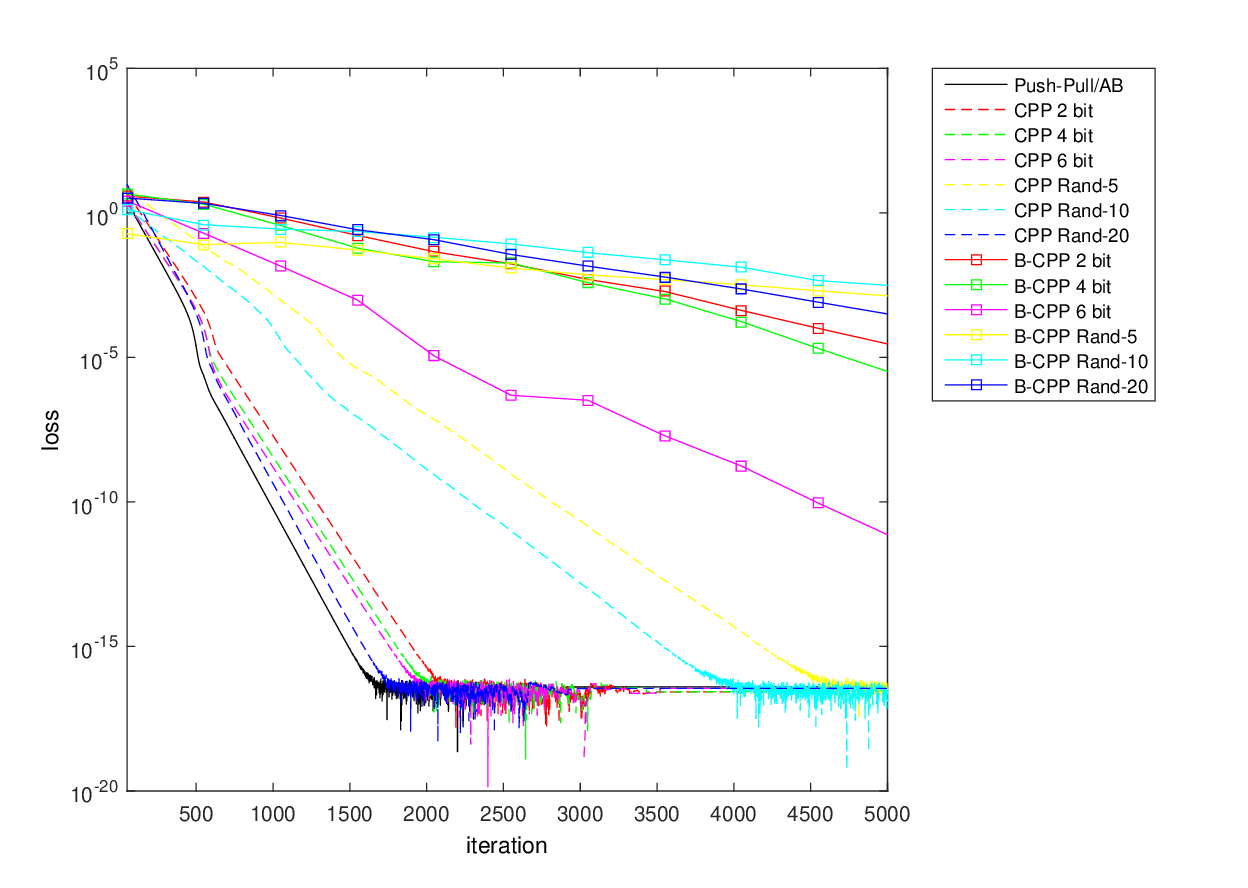}
    \caption{Linear convergence of Push-Pull/$\calA\calB$, CPP, and B-CPP with $b$ bit quantization ($b = 2, 4, 6$) and Rand-k ($k = 5, 10, 20$) compressors. }
    \label{fig:fig1}

\end{figure}
\begin{figure}[!ht]
    \centering
    \includegraphics[scale=0.4]{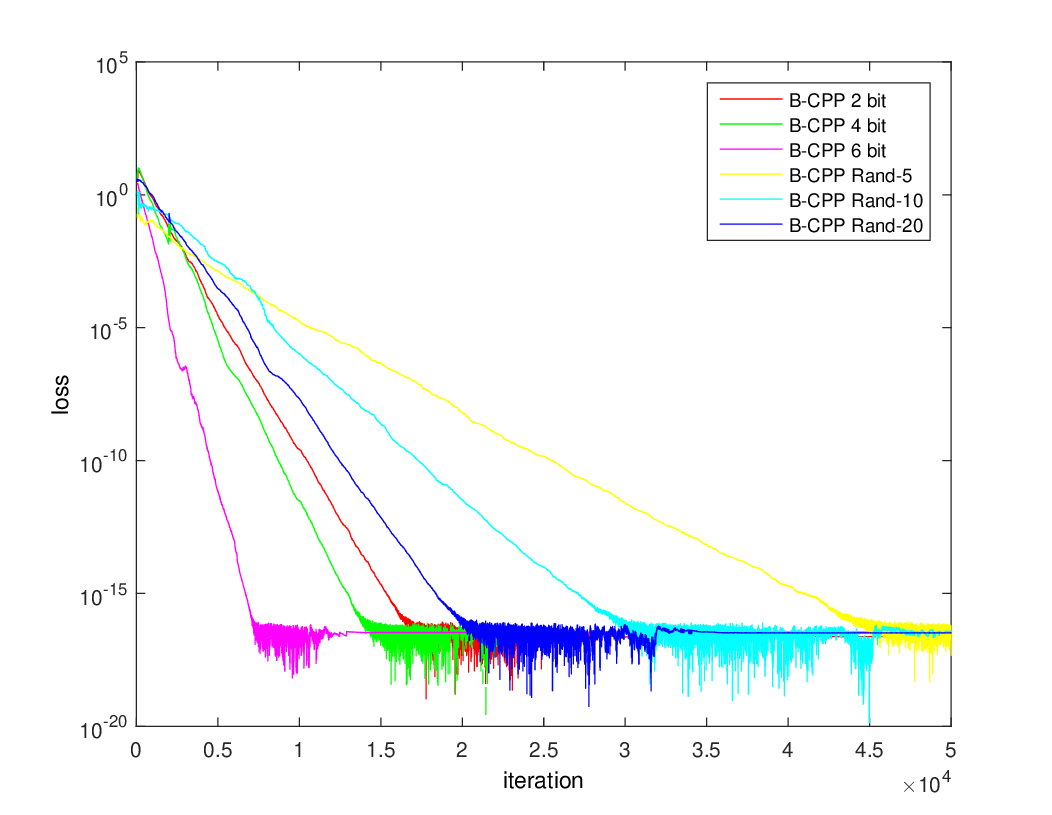}
    \caption{Linear convergence of B-CPP with $b$ bit quantization ($b = 2, 4, 6$) and Rand-k ($k = 5, 10, 20$) compressors. }
    \label{fig:BCPP}

\end{figure}
It can be seen from Fig.~\ref{fig:fig1} and Fig.~\ref{fig:BCPP} that all the trajectories exhibit linear convergence, and the exact Push-Pull/$\calA\calB$ method is faster than CPP and B-CPP with any compression methods.
This is reasonable as the compression operator induces additional errors compared to the exact method, and these additional errors could slow down the convergence.
Meanwhile, as the values of $b$ or $k$ increases, both CPP and B-CPP speed up since the compression errors decrease.
When $b = 6$ or $k = 20$, the trajectories of CPP are very close to that of exact Push-Pull/$\calA\calB$, which indicates that when the compression errors are small, they are no longer the bottleneck of convergence.

Within the same number of iterations, CPP outperforms B-CPP,
and the trajectories of CPP are smoother than B-CPP.
These results can be expected since CPP updates all the local variables in a single iteration, while in B-CPP, only one node updates with its neighbors.
To guarantee the linear convergence, B-CPP requires smaller parameters and stepsizes.
In addition, the mixing matrix at each iteration of B-CPP can be regarded as a stochastic matrix, which will induce additional variances.

\subsection{Communication efficiency}
When we compare the number of transmitted bits to reach certain levels of accuracy, CPP and B-CPP show their superiority. We consider the same settings as in Section \ref{subsec: numeric_linear}.
We can see from all of the sub-figures of Fig.~\ref{fig:comm} that, to reach a high accuracy within about $10^{-15}$, the number of transmitted bits required by these methods have the ranking: B-CPP $<$ CPP $<$ Push-Pull/$\calA\calB$.

\begin{figure}[!ht]
    \centering
    \includegraphics[scale=0.6]{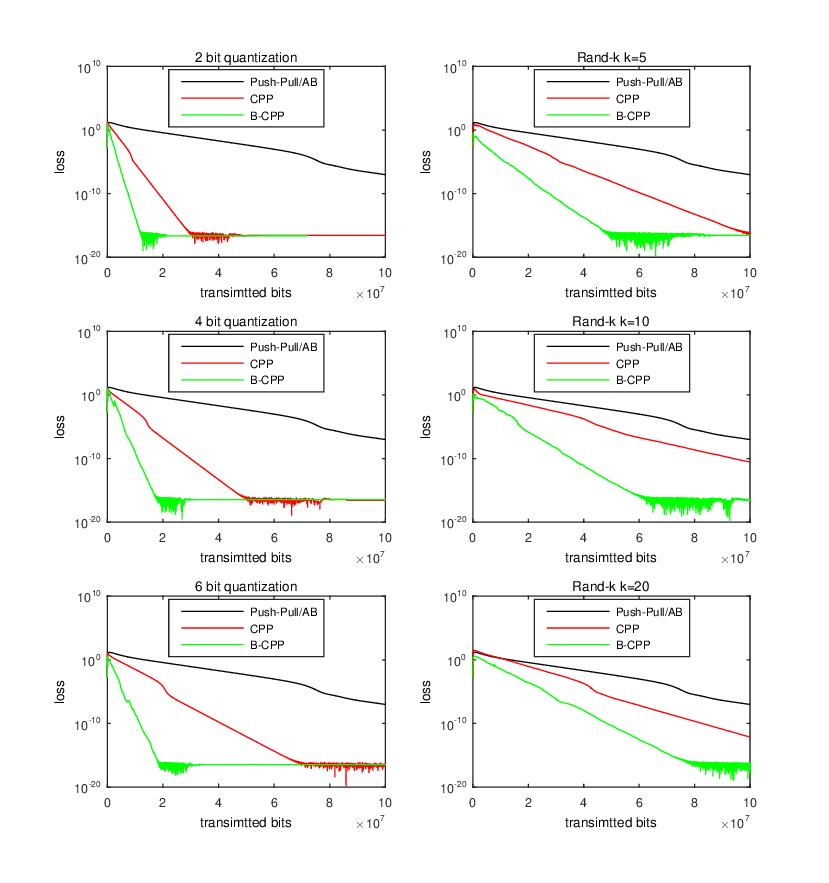}
    \caption{Performance of Push-Pull/$\calA\calB$, CPP, B-CPP against the number of transmitted bits: the left column shows the results with quantization ($b = 2, 4, 6$) and the right column shows the results with Rand-k ($k = 5, 10, 20$). }
    \label{fig:comm}

\end{figure}

To see why CPP outperforms Push-Pull/$\calA\calB$, note that  the vectors sent in CPP have been compressed, and hence the transmitted bits at each iteration are greatly reduced compared to Push-Pull/$\calA\calB$.
Although the additional compression errors slow down the convergence,
our design for CPP guarantees that the impact on the convergence rate is relatively small. Therefore, CPP is much more communication-efficient than the exact Push-Pull/$\calA\calB$ method.
Moreover, in all cases, B-CPP is much more communication-efficient than CPP.
This is because when CPP converges, each agent will receive similar, or ``overlapping" vectors from different neighbors,
which impacts the communication-efficiency.
By contrast, for B-CPP, only one agent wakes up in each iteration, and its information can be utilized by its neighbors more efficiently.

It is worth noting that for both CPP and B-CPP, the choices $b = 2$ for quantization or $k = 5$ for Rand-k are more communication-efficient than $b = 4, 6$ or $k = 10, 20$.
This indicates that as the compression accuracy becomes smaller, its impact exhibits ``marginal effects”.
In other words, when the compression errors are not the bottleneck for the convergence, sacrificing the communication costs for faster convergence will reduce the communication efficiency.

\subsection{Different topologies}
We also examine the performance of CPP and B-CPP on communication networks with different levels of connectivity.

\begin{figure}[!ht]
    \centering
    \includegraphics[scale=0.35]{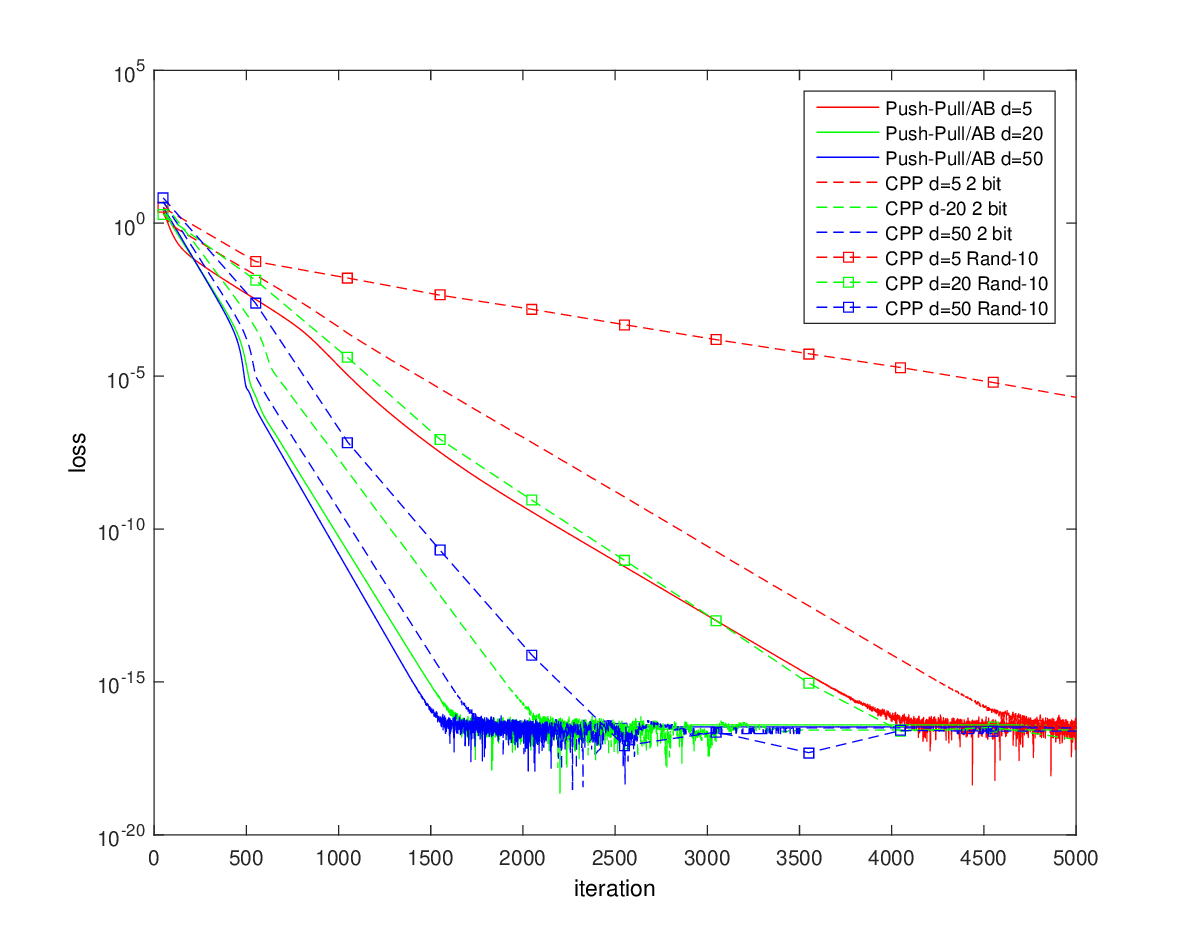}
    \caption{Performance of CPP and Push-Pull/$\calA\calB$ with different communication networks under both quantization and Rand-k compressors. 
     }
    \label{fig:CPPg}

\end{figure}

\begin{figure}[!hb]
    \centering
    \includegraphics[scale=0.41]{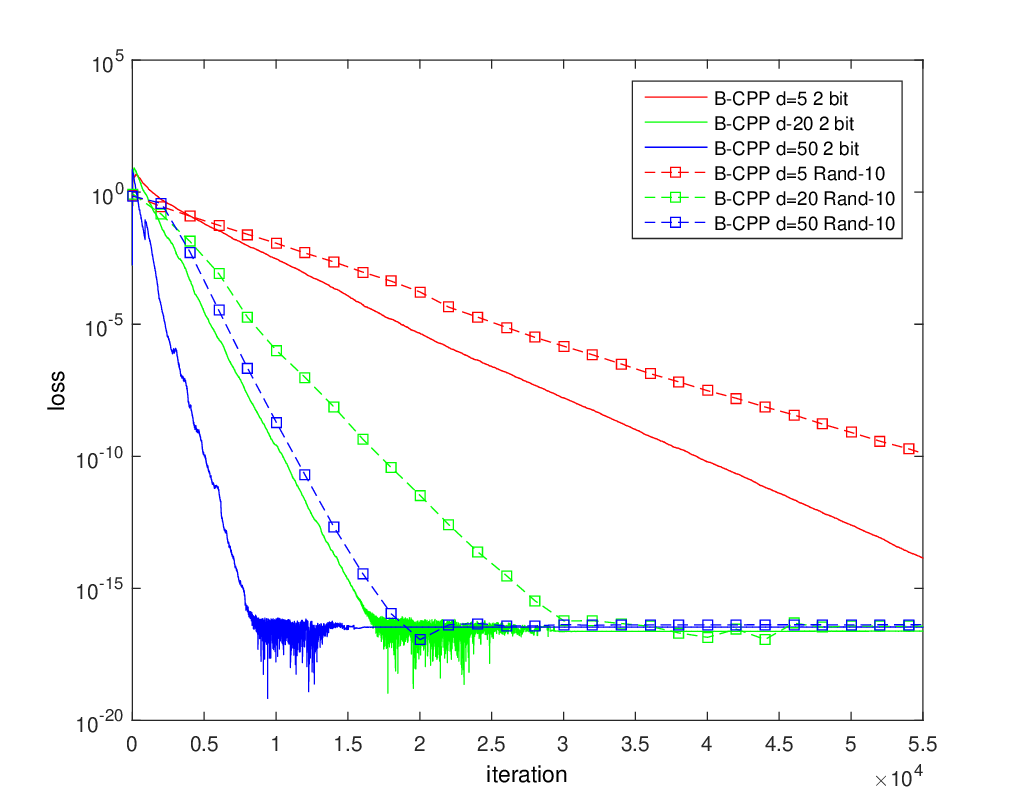}
    \caption{Performance of B-CPP with different communication networks under both quantization and Rand-k compressors. }
    \label{fig:BCPPg}

\end{figure}

We consider $d = 5, 10, 20$ respectively:
as the value of $d$ increases, the communication network has a better connectivity.
As before, both the quantization and the Rand-k compressors are considered, and we show the performance of CPP and B-CPP separately for better clarity.

In Fig.~\ref{fig:CPPg} and Fig.~\ref{fig:BCPPg}, we can see that when $d$ increases, the convergence of both CPP/B-CPP and Push-Pull/$\calA\calB$ speed up.
This is expected since better connectivity implies less consensus errors in each iteration, and the algorithms perform closer to the centralized gradient descent algorithm which is faster.







}

\section{Conclusions}
\label{sec:conclusions}

In this paper, we proposed two communication-efficient algorithms for decentralized optimization over a multi-agent network with general directed topology. First, we consider a novel communication-efficient gradient tracking based method, termed CPP, that combines the Push-Pull method with communication compression. CPP can be applied to a general class of unbiased compression operators and achieves linear convergence for strongly convex and smooth objective functions.
Second, we consider a broadcast-like version of CPP (B-CPP) which also achieves linear convergence rate for strongly convex and smooth objective functions. B-CPP can be applied in an asynchronous broadcast setting and further reduce communication costs compared to CPP.

\ifCLASSOPTIONcaptionsoff
  \newpage
\fi



%

\bibliographystyle{IEEEtran}
\bibliography{IEEEabrv,ref}

\begin{thebibliography}{10}
\providecommand{\url}[1]{#1}
\csname url@samestyle\endcsname
\providecommand{\newblock}{\relax}
\providecommand{\bibinfo}[2]{#2}
\providecommand{\BIBentrySTDinterwordspacing}{\spaceskip=0pt\relax}
\providecommand{\BIBentryALTinterwordstretchfactor}{4}
\providecommand{\BIBentryALTinterwordspacing}{\spaceskip=\fontdimen2\font plus
\BIBentryALTinterwordstretchfactor\fontdimen3\font minus
  \fontdimen4\font\relax}
\providecommand{\BIBforeignlanguage}[2]{{%
\expandafter\ifx\csname l@#1\endcsname\relax
\typeout{** WARNING: IEEEtran.bst: No hyphenation pattern has been}%
\typeout{** loaded for the language `#1'. Using the pattern for}%
\typeout{** the default language instead.}%
\else
\language=\csname l@#1\endcsname
\fi
#2}}
\providecommand{\BIBdecl}{\relax}
\BIBdecl

\bibitem{cohen2017projected}
K.~Cohen, A.~Nedi{\'c}, and R.~Srikant, ``On projected stochastic gradient
  descent algorithm with weighted averaging for least squares regression,''
  \emph{IEEE Transactions on Automatic Control}, vol.~62, no.~11, pp.
  5974--5981, 2017.

\bibitem{forrester2007multi}
A.~I. Forrester, A.~S{\'o}bester, and A.~J. Keane, ``Multi-fidelity
  optimization via surrogate modelling,'' in \emph{Proceedings of the Royal
  Society of London A: Mathematical, Physical and Engineering Sciences}, vol.
  463.\hskip 1em plus 0.5em minus 0.4em\relax The Royal Society, 2007, pp.
  3251--3269.

\bibitem{nedic2017fast}
A.~Nedi{\'c}, A.~Olshevsky, and C.~A. Uribe, ``Fast convergence rates for
  distributed non-{B}ayesian learning,'' \emph{IEEE Transactions on Automatic
  Control}, vol.~62, no.~11, pp. 5538--5553, 2017.

\bibitem{chen2012diffusion}
J.~Chen and A.~H. Sayed, ``Diffusion adaptation strategies for distributed
  optimization and learning over networks,'' \emph{IEEE Transactions on Signal
  Processing}, vol.~60, no.~8, pp. 4289--4305, 2012.

\bibitem{pu2016noise}
S.~Pu, A.~Garcia, and Z.~Lin, ``Noise reduction by swarming in social
  foraging,'' \emph{IEEE Transactions on Automatic Control}, vol.~61, no.~12,
  pp. 4007--4013, 2016.

\bibitem{baingana2014proximal}
B.~Baingana, G.~Mateos, and G.~B. Giannakis, ``Proximal-gradient algorithms for
  tracking cascades over social networks,'' \emph{IEEE Journal of Selected
  Topics in Signal Processing}, vol.~8, no.~4, pp. 563--575, 2014.

\bibitem{cohen2017distributed}
K.~Cohen, A.~Nedi{\'c}, and R.~Srikant, ``Distributed learning algorithms for
  spectrum sharing in spatial random access wireless networks,'' \emph{IEEE
  Transactions on Automatic Control}, vol.~62, no.~6, pp. 2854--2869, 2017.

\bibitem{mateos2012distributed}
G.~Mateos and G.~B. Giannakis, ``Distributed recursive least-squares: Stability
  and performance analysis,'' \emph{IEEE Transactions on Signal Processing},
  vol.~60, no.~7, pp. 3740--3754, 2012.

\bibitem{nedic2018distributed}
A.~Nedi{\'c} and J.~Liu, ``Distributed optimization for control,'' \emph{Annual
  Review of Control, Robotics, and Autonomous Systems}, vol.~1, pp. 77--103,
  2018.

\bibitem{nedic2009distributed}
A.~Nedic and A.~Ozdaglar, ``Distributed subgradient methods for multi-agent
  optimization,'' \emph{IEEE Transactions on Automatic Control}, vol.~54,
  no.~1, pp. 48--61, 2009.

\bibitem{li2019decentralized}
Z.~Li, W.~Shi, and M.~Yan, ``A decentralized proximal-gradient method with
  network independent step-sizes and separated convergence rates,'' \emph{IEEE
  Transactions on Signal Processing}, vol.~67, no.~17, pp. 4494--4506, 2019.

\bibitem{qu2017harnessing}
G.~Qu and N.~Li, ``Harnessing smoothness to accelerate distributed
  optimization,'' \emph{IEEE Transactions on Control of Network Systems},
  vol.~5, no.~3, pp. 1245--1260, 2017.

\bibitem{scaman2018optimal}
K.~Scaman, F.~Bach, S.~Bubeck, L.~Massouli{\'e}, and Y.~T. Lee, ``Optimal
  algorithms for non-smooth distributed optimization in networks,'' in
  \emph{Advances in Neural Information Processing Systems}, 2018, pp.
  2745--2754.

\bibitem{shi2015extra}
W.~Shi, Q.~Ling, G.~Wu, and W.~Yin, ``{EXTRA}: An exact first-order algorithm
  for decentralized consensus optimization,'' \emph{SIAM Journal on
  Optimization}, vol.~25, no.~2, pp. 944--966, 2015.

\bibitem{uribe2017optimal}
\BIBentryALTinterwordspacing
C.~A. Uribe, S.~Lee, A.~Gasnikov, and A.~Nedić, ``A dual approach for optimal
  algorithms in distributed optimization over networks,'' \emph{Optimization
  Methods and Software}, vol.~36, no.~1, pp. 171--210, 2021. [Online].
  Available: \url{https://doi.org/10.1080/10556788.2020.1750013}
\BIBentrySTDinterwordspacing

\bibitem{xu2015augmented}
J.~Xu, S.~Zhu, Y.~C. Soh, and L.~Xie, ``Augmented distributed gradient methods
  for multi-agent optimization under uncoordinated constant stepsizes,'' in
  \emph{2015 54th IEEE Conference on Decision and Control (CDC)}.\hskip 1em
  plus 0.5em minus 0.4em\relax IEEE, 2015, pp. 2055--2060.

\bibitem{kempe2003gossip}
D.~Kempe, A.~Dobra, and J.~Gehrke, ``Gossip-based computation of aggregate
  information,'' in \emph{44th Annual IEEE Symposium on Foundations of Computer
  Science, 2003. Proceedings.}\hskip 1em plus 0.5em minus 0.4em\relax IEEE,
  2003, pp. 482--491.

\bibitem{nedic2014distributed}
A.~Nedi{\'c} and A.~Olshevsky, ``Distributed optimization over time-varying
  directed graphs,'' \emph{IEEE Transactions on Automatic Control}, vol.~60,
  no.~3, pp. 601--615, 2014.

\bibitem{xi2017dextra}
C.~Xi and U.~A. Khan, ``Dextra: A fast algorithm for optimization over directed
  graphs,'' \emph{IEEE Transactions on Automatic Control}, vol.~62, no.~10, pp.
  4980--4993, 2017.

\bibitem{zeng2015extrapush}
J.~Zeng and W.~Yin, ``Extra{P}ush for convex smooth decentralized optimization
  over directed networks,'' \emph{Journal of Computational Mathematics},
  vol.~35, no.~4, pp. 383--396, 2017.

\bibitem{nedic2017achieving}
A.~Nedic, A.~Olshevsky, and W.~Shi, ``Achieving geometric convergence for
  distributed optimization over time-varying graphs,'' \emph{SIAM Journal on
  Optimization}, vol.~27, no.~4, pp. 2597--2633, 2017.

\bibitem{tian2018asy}
Y.~Tian, Y.~Sun, and G.~Scutari, ``{ASY-SONATA}: Achieving linear convergence
  in distributed asynchronous multiagent optimization,'' in \emph{2018 56th
  Annual Allerton Conference on Communication, Control, and Computing
  (Allerton)}.\hskip 1em plus 0.5em minus 0.4em\relax IEEE, 2018, pp. 543--551.

\bibitem{xi2017add}
C.~Xi, R.~Xin, and U.~A. Khan, ``{ADD-OPT}: Accelerated distributed directed
  optimization,'' \emph{IEEE Transactions on Automatic Control}, vol.~63,
  no.~5, pp. 1329--1339, 2017.

\bibitem{pu2020push}
S.~Pu, W.~Shi, J.~Xu, and A.~Nedic, ``Push-pull gradient methods for
  distributed optimization in networks,'' \emph{IEEE Transactions on Automatic
  Control}, 2020.

\bibitem{xin2018linear}
R.~Xin and U.~A. Khan, ``A linear algorithm for optimization over directed
  graphs with geometric convergence,'' \emph{IEEE Control Systems Letters},
  vol.~2, no.~3, pp. 315--320, 2018.

\bibitem{saadatniaki2020decentralized}
F.~Saadatniaki, R.~Xin, and U.~A. Khan, ``Decentralized optimization over
  time-varying directed graphs with row and column-stochastic matrices,''
  \emph{IEEE Transactions on Automatic Control}, vol.~65, no.~11, pp.
  4769--4780, 2020.

\bibitem{zhangfully}
J.~Zhang and K.~You, ``Fully asynchronous distributed optimization with linear
  convergence in directed networks,'' \emph{arXiv preprint arXiv:1901.08215},
  2021.

\bibitem{alis}
D.~Alistarh, D.~Grubic, J.~Li, R.~Tomioka, and M.~Vojnovic, ``{QSGD}:
  Communication-efficient {SGD} via gradient quantization and encoding,''
  \emph{Advances in Neural Information Processing Systems}, pp. 1709--1720,
  2017.

\bibitem{bernstein2018signsgd}
J.~Bernstein, Y.-X. Wang, K.~Azizzadenesheli, and A.~Anandkumar, ``signsgd:
  Compressed optimisation for non-convex problems,'' in \emph{International
  Conference on Machine Learning}.\hskip 1em plus 0.5em minus 0.4em\relax PMLR,
  2018, pp. 560--569.

\bibitem{beznosikov2020biased}
A.~Beznosikov, S.~Horv{\'a}th, P.~Richt{\'a}rik, and M.~Safaryan, ``On biased
  compression for distributed learning,'' \emph{arXiv preprint
  arXiv:2002.12410}, 2020.

\bibitem{karimireddy2019error}
S.~P. Karimireddy, Q.~Rebjock, S.~Stich, and M.~Jaggi, ``Error feedback fixes
  signsgd and other gradient compression schemes,'' in \emph{International
  Conference on Machine Learning}.\hskip 1em plus 0.5em minus 0.4em\relax PMLR,
  2019, pp. 3252--3261.

\bibitem{liu2020distributed}
J.~Liu, C.~Zhang \emph{et~al.}, ``Distributed learning systems with first-order
  methods,'' \emph{Foundations and Trends{\textregistered} in Databases},
  vol.~9, no.~1, pp. 1--100, 2020.

\bibitem{liu2020double}
X.~Liu, Y.~Li, J.~Tang, and M.~Yan, ``A double residual compression algorithm
  for efficient distributed learning,'' in \emph{International Conference on
  Artificial Intelligence and Statistics}.\hskip 1em plus 0.5em minus
  0.4em\relax PMLR, 2020, pp. 133--143.

\bibitem{mishchenko2019distributed}
K.~Mishchenko, E.~Gorbunov, M.~Tak{\'a}{\v{c}}, and P.~Richt{\'a}rik,
  ``Distributed learning with compressed gradient differences,'' \emph{arXiv
  preprint arXiv:1901.09269}, 2019.

\bibitem{seide20141}
F.~Seide, H.~Fu, J.~Droppo, G.~Li, and D.~Yu, ``1-bit stochastic gradient
  descent and its application to data-parallel distributed training of speech
  dnns,'' in \emph{Fifteenth Annual Conference of the International Speech
  Communication Association}, 2014.

\bibitem{stich2020communication}
S.~U. Stich, ``On communication compression for distributed optimization on
  heterogeneous data,'' \emph{arXiv preprint arXiv:2009.02388}, 2020.

\bibitem{stich2018sparsified}
S.~U. Stich, J.-B. Cordonnier, and M.~Jaggi, ``Sparsified {SGD} with memory,''
  in \emph{Advances in Neural Information Processing Systems}, 2018, pp.
  4447--4458.

\bibitem{tang2019doublesqueeze}
H.~Tang, C.~Yu, X.~Lian, T.~Zhang, and J.~Liu, ``Doublesqueeze: Parallel
  stochastic gradient descent with double-pass error-compensated compression,''
  in \emph{International Conference on Machine Learning}.\hskip 1em plus 0.5em
  minus 0.4em\relax PMLR, 2019, pp. 6155--6165.

\bibitem{xu2020compressed}
H.~Xu, C.-Y. Ho, A.~M. Abdelmoniem, A.~Dutta, E.~H. Bergou, K.~Karatsenidis,
  M.~Canini, and P.~Kalnis, ``Compressed communication for distributed deep
  learning: Survey and quantitative evaluation,'' King Abdullah University of
  Science and Technology (KAUST), Tech. Rep., 2020.

\bibitem{lin2017deep}
\BIBentryALTinterwordspacing
Y.~Lin, S.~Han, H.~Mao, Y.~Wang, and B.~Dally, ``Deep gradient compression:
  Reducing the communication bandwidth for distributed training,'' in
  \emph{International Conference on Learning Representations}, 2018. [Online].
  Available: \url{https://openreview.net/forum?id=SkhQHMW0W}
\BIBentrySTDinterwordspacing

\bibitem{xiao2005scheme}
L.~Xiao, S.~Boyd, and S.~Lall, ``A scheme for robust distributed sensor fusion
  based on average consensus,'' in \emph{IPSN 2005. Fourth International
  Symposium on Information Processing in Sensor Networks, 2005.}\hskip 1em plus
  0.5em minus 0.4em\relax IEEE, 2005, pp. 63--70.

\bibitem{carli2007average}
R.~Carli, F.~Fagnani, P.~Frasca, T.~Taylor, and S.~Zampieri, ``Average
  consensus on networks with transmission noise or quantization,'' in
  \emph{2007 European Control Conference (ECC)}.\hskip 1em plus 0.5em minus
  0.4em\relax IEEE, 2007, pp. 1852--1857.

\bibitem{nedic2008distributed}
A.~Nedic, A.~Olshevsky, A.~Ozdaglar, and J.~N. Tsitsiklis, ``Distributed
  subgradient methods and quantization effects,'' in \emph{2008 47th IEEE
  Conference on Decision and Control}.\hskip 1em plus 0.5em minus 0.4em\relax
  IEEE, 2008, pp. 4177--4184.

\bibitem{aysal2008distributed}
T.~C. Aysal, M.~J. Coates, and M.~G. Rabbat, ``Distributed average consensus
  with dithered quantization,'' \emph{IEEE Transactions on Signal Processing},
  vol.~56, no.~10, pp. 4905--4918, 2008.

\bibitem{carli2010gossip}
R.~Carli, F.~Fagnani, P.~Frasca, and S.~Zampieri, ``Gossip consensus algorithms
  via quantized communication,'' \emph{Automatica}, vol.~46, no.~1, pp. 70--80,
  2010.

\bibitem{yuan2012distributed}
D.~Yuan, S.~Xu, H.~Zhao, and L.~Rong, ``Distributed dual averaging method for
  multi-agent optimization with quantized communication,'' \emph{Systems \&
  Control Letters}, vol.~61, no.~11, pp. 1053--1061, 2012.

\bibitem{reisizadeh2019exact}
A.~Reisizadeh, A.~Mokhtari, H.~Hassani, and R.~Pedarsani, ``An exact quantized
  decentralized gradient descent algorithm,'' \emph{IEEE Transactions on Signal
  Processing}, vol.~67, no.~19, pp. 4934--4947, 2019.

\bibitem{carli2010quantized}
R.~Carli, F.~Bullo, and S.~Zampieri, ``Quantized average consensus via dynamic
  coding/decoding schemes,'' \emph{International Journal of Robust and
  Nonlinear Control: IFAC-Affiliated Journal}, vol.~20, no.~2, pp. 156--175,
  2010.

\bibitem{doan2018accelerating}
T.~T. Doan, S.~T. Maguluri, and J.~Romberg, ``Fast convergence rates of
  distributed subgradient methods with adaptive quantization,'' \emph{IEEE
  Transactions on Automatic Control}, vol.~66, no.~5, pp. 2191--2205, 2020.

\bibitem{berahas2019nested}
A.~S. Berahas, C.~Iakovidou, and E.~Wei, ``Nested distributed gradient methods
  with adaptive quantized communication,'' in \emph{2019 IEEE 58th Conference
  on Decision and Control (CDC)}.\hskip 1em plus 0.5em minus 0.4em\relax IEEE,
  2019, pp. 1519--1525.

\bibitem{li2020acceleration}
Z.~Li, D.~Kovalev, X.~Qian, and P.~Richtarik, ``Acceleration for compressed
  gradient descent in distributed and federated optimization,'' in
  \emph{International Conference on Machine Learning}.\hskip 1em plus 0.5em
  minus 0.4em\relax PMLR, 2020, pp. 5895--5904.

\bibitem{liu2020linear}
\BIBentryALTinterwordspacing
X.~Liu, Y.~Li, R.~Wang, J.~Tang, and M.~Yan, ``Linear convergent decentralized
  optimization with compression,'' in \emph{International Conference on
  Learning Representations}, 2021. [Online]. Available:
  \url{https://openreview.net/forum?id=84gjULz1t5}
\BIBentrySTDinterwordspacing

\bibitem{koloskova2019decentralized}
A.~Koloskova, S.~Stich, and M.~Jaggi, ``Decentralized stochastic optimization
  and gossip algorithms with compressed communication,'' in \emph{International
  Conference on Machine Learning}.\hskip 1em plus 0.5em minus 0.4em\relax PMLR,
  2019, pp. 3478--3487.

\bibitem{koloskova2019decentralizeda}
\BIBentryALTinterwordspacing
A.~Koloskova*, T.~Lin*, S.~U. Stich, and M.~Jaggi, ``Decentralized deep
  learning with arbitrary communication compression,'' in \emph{International
  Conference on Learning Representations}, 2020. [Online]. Available:
  \url{https://openreview.net/forum?id=SkgGCkrKvH}
\BIBentrySTDinterwordspacing

\bibitem{tang2018communication}
H.~Tang, S.~Gan, C.~Zhang, T.~Zhang, and J.~Liu, ``Communication compression
  for decentralized training,'' in \emph{Advances in Neural Information
  Processing Systems}, 2018, pp. 7663--7673.

\bibitem{tang2019deepsqueeze}
H.~Tang, X.~Lian, S.~Qiu, L.~Yuan, C.~Zhang, T.~Zhang, and J.~Liu,
  ``Deepsqueeze: Decentralization meets error-compensated compression,''
  \emph{arXiv preprint arXiv:1907.07346}, 2019.

\bibitem{kajiyama2020linear}
Y.~Kajiyama, N.~Hayashi, and S.~Takai, ``Linear convergence of consensus-based
  quantized optimization for smooth and strongly convex cost functions,''
  \emph{IEEE Transactions on Automatic Control}, 2020.

\bibitem{li2021compressed}
Z.~Li, Y.~Liao, K.~Huang, and S.~Pu, ``Compressed gradient tracking for
  decentralized optimization with linear convergence,'' \emph{arXiv preprint
  arXiv:2103.13748}, 2021.

\bibitem{xiong2021quantized}
Y.~Xiong, L.~Wu, K.~You, and L.~Xie, ``Quantized distributed gradient tracking
  algorithm with linear convergence in directed networks,'' \emph{arXiv
  preprint arXiv:2104.03649}, 2021.

\bibitem{zhang2021innovation}
J.~Zhang, K.~You, and L.~Xie, ``Innovation compression for
  communication-efficient distributed optimization with linear convergence,''
  \emph{arXiv preprint arXiv:2105.06697}, 2021.

\bibitem{aysal2009broadcast}
T.~C. Aysal, M.~E. Yildiz, A.~D. Sarwate, and A.~Scaglione, ``Broadcast gossip
  algorithms for consensus,'' \emph{IEEE Transactions on Signal processing},
  vol.~57, no.~7, pp. 2748--2761, 2009.

\bibitem{boyd2006randomized}
S.~Boyd, A.~Ghosh, B.~Prabhakar, and D.~Shah, ``Randomized gossip algorithms,''
  \emph{IEEE Transactions on Information Theory}, vol.~52, no.~6, pp.
  2508--2530, 2006.

\bibitem{pu2020distributed}
S.~Pu and A.~Nedi{\'c}, ``Distributed stochastic gradient tracking methods,''
  \emph{Mathematical Programming}, pp. 1--49, 2020.

\bibitem{horn2012matrix}
R.~A. Horn and C.~R. Johnson, \emph{Matrix Analysis}.\hskip 1em plus 0.5em
  minus 0.4em\relax Cambridge University Press, 2012.

\bibitem{berman1994nonnegative}
A.~Berman and R.~J. Plemmons, \emph{Nonnegative Matrices in the Mathematical
  Sciences}.\hskip 1em plus 0.5em minus 0.4em\relax Society for Industrial and
  Applied Mathematic, 1994.

\bibitem{mansouri2013quantitative}
K.~Mansouri, T.~Ringsted, D.~Ballabio, R.~Todeschini, and V.~Consonni,
  ``Quantitative structure--activity relationship models for ready
  biodegradability of chemicals,'' \emph{Journal of chemical information and
  modeling}, vol.~53, no.~4, pp. 867--878, 2013.

\end{thebibliography}

%

%
%
%





\newpage

\onecolumn
\def\tmptitle{Supplementary Material}
\etocsettocstyle{\Large\textbf{\tmptitle}\par\normalsize
}{\bigskip}
\etocdepthtag.toc{mtappendix}
\etocsettagdepth{mtchapter}{none}
\etocsettagdepth{mtappendix}{subsection}
\tableofcontents

\renewcommand*{\thesection}{\Alph{section}}
\setcounter{section}{19}
\section*{}
\subsection{Proof of Lemma~\ref{def nR nC}}\label{lem: proof of def nR nC}
We prove the existence of $\Rtil,\theR,\traRt$ for a given $\RR$, and the existence of $\Ctil,\theC,\traCt$ can be shown similarly.

Under Assumption~\ref{asp: R and C}, all eigenvalues of $\RR$ lie in the unit circle, and $1$ is the only eigenvalue with modulus one. In addition, the multiplicity of eigenvalue $1$ is one.
Denote all eigenvalues of $\RR$ by $z_1 = 1, z_2, \cdots, z_n$.
Then, $\abs{z_k} < 1$ for all $2\leq k\leq n$.

There is an invertible matrix $\DD$ such that $\DD\inv\RR\DD = \JJ$, where $\JJ$ is the Jordan form of $\RR$.
The columns of $\DD$ are right eigenvectors or generalized right eigenvectors of $\DD$.
We can rearrange the columns of $\DD$ such that the unique $1$-by-$1$ Jordan block of eigenvalue $1$ is on the left-top of $\JJ$, then, the first column of $\DD$ is parallel to the all-ones vector $\one$.

{We denote $\uu = \vecr$ in this proof for simplicity. }
As eigenvalue $1$ has multiplicity one in the spectral of $\RR $ and $\uu$ is the left eigenvalue with respect to $1$, all other columns of $\DD$ except the first column are orthogonal to $\uu$.
Thus, $\frac{1}{n}\pr{\one\uu\tp}\DD = \DD\Diag{1, 0, \cdots, 0}$, i.e., $\DD\inv\frac{\one\uu\tp}{n}\DD = \Diag{1, 0, \cdots, 0}$.

Decompose the Jordan form as $\JJ = \XX + \YY$, where $\XX = \Diag{1, z_2, \cdots, z_n}$ contains the diagonal of $\JJ$ and $\YY$ contains the superdiagonal.
Define $\VV_t = \Diag{1, t, \cdots, t^n}$ for $t > 0$.
Then, $\VV_t\inv \JJ \VV_t = \XX + t\YY$.

If $\YY \neq \zero$, take $t \leq \min_{2\leq k\leq n} \frac{1 - \abs{z_k}}{2\nt{\YY}} $; if $\YY =
\zero$, take $t = 1$.
By letting $\theR = \min_{2\leq k\leq n} \pr{1 - \abs{z_k}}/2$, we have
\seq{
     &\nt{\VV_t\inv\DD\inv\pr{\RR - \frac{1}{n}\one\uu\tp}\DD\VV_t}
    \leq   \nt{\Diag{0, z_2, \cdots, z_n}} + t \nt{\YY}
       = \max_{2\leq k\leq n} \abs{z_k} + t \nt{\YY}
       \leq 1 - \theR .
}
The equality $\nt{\VV_t\inv\DD\inv\pr{\II - \frac{1}{n}\one\uu\tp}\DD\VV_t} = 1 $ is straightforward.
Thus,
\seq{
  &\nt{\VV_t\inv\DD\inv\Rb\DD\VV_t}
  \leq \pr{1 - \beta} \nt{\VV_t\inv\DD\inv\pr{\II - \frac{1}{n}\one\uu\tp}\DD\VV_t} + \beta \nt{\VV_t\inv\DD\inv\pr{\RR - \frac{1}{n}\one\uu\tp}\DD\VV_t} \leq 1 - \theR \beta.
}

By the choice of $t$, we have $t\nt{\YY} < 1$,
then
\seq{&\nt{\VV_t\inv\DD\inv\RR\DD\VV_t} \leq \max_{1\leq k\leq n}\abs{z_k} + t\nt{\YY} \leq 1 + t\nt{\YY} \leq 2,}
and \seq{&\nt{\VV_t\inv\DD\inv\pr{\RR - \II}\DD\VV_t}
\leq \nt{\VV_t\inv\DD\inv\RR\DD\VV_t} + 1 \leq 3.}

The inequality (\ref{eq: nR Rb nC Cg}) as well as $\nR{\RR} \leq 2$, $\nR{\RR - \II} \leq 3$ and $\nt{\xx} \leq \nR{\xx}$ follows by letting $\Rtil =  \nt{\VV_t\inv\DD\inv}\inv \VV_t\inv\DD\inv$.
And $\mt{\AA} \leq \mR{\AA}$ follows by Definition~\ref{def: induced matrix norm n by p} and $\nt{\xx} \leq \nR{\xx}$.
By taking $\traRt = \nt{\Rtil}$,  we have $\nR{\xx} \leq \traRt \nt{\xx}$ for any vector $\xx\in \Real^n$.
\qed

\newcommand\uxa{\calZ}

\subsection{Proof of Lemma~\ref{lem:cpushpullmatineq}}\label{sec:proofofCpushpullmatineq1}
By Assumption~\ref{assp: compress},
\seql{\label{eq: bound for Ek}}{
        &\expec{\mt{\EX^k}^2 | \calF_k} \leq \cerr{2}\mt{\XX^k - \UU^k}^2,\ \expec{\mt{\EE^k}^2 | \calF^+_k} \leq \cerr{2} \mt{\YY^k - \VV^k}^2 .
}
Next, we provide bounds for the quantities $\expec{\nt{\xa^{k+1} - \xx^*}}, \expec{\mR{\PR\XX^k}^2}, \expec{\mC{\PC\YY^k}^2}$, $\expec{\mt{\UU^{k+1} - \XX^{k+1}}^2}$ and $\expec{\mt{\YY^k}^2}$ in turn.

Bounding $\expec{\nt{\xa^{k+1} - \xx^*}}$: \
we first analyze the terms in (\ref{eq: xa expansion}) one-by-one.
By Lemma 10 in~\cite{qu2017harnessing} and Assumption~\ref{assp: require each fi}, for $\altil \leq \frac{2}{\mu + L}$,
\seql{\label{eq: gradient norm small}}{
  \nt{\xa^k - \altil \gg^k - \xx^*} \leq \pr{1 - \altil \mu} \nt{\xa^k - \xx^*}.
}
By Assumption~\ref{assp: require each fi} and the fact $\mt{\cdot} \leq \mR{\cdot}$ we assumed in Lemma~\ref{def nR nC},
\seql{\label{eq: dd gg in xa expansion}}{
        &\nt{\gg^k - \ya^k} = \frac{1}{n} \nt{\one\tp\na\FF\pr{\one\xa^k} - \one\tp\na\FF\pr{\XX^k} }  
        \leq \frac{1}{n}\sum_{i=1}^{n}\nt{\na f_i\pr{\xa^k} - \na f_i\pr{\xx_i^k}}  \\  
        \leq& \frac{L}{n}\sum_{i=1}^{n}\nt{\xa^k - \xx_i^k} 
        \leq \frac{L}{\sqrt{n}} \mt{\one\xa^k - \XX^k} \leq \frac{L}{\sqrt{n}} \mR{\PR\XX^k}.
}

Taking conditional expectation on both sides of (\ref{eq: xa expansion}) yields
\seql{\label{eq: xa square expec 1}}{
        &\expec{\nt{\xa^{k+1} - \xx^*}^2 | \calF_k}  \\
        =& \nt{\xa^k - \altil \gg^k - \xx^* + \altil\pr{\gg^k - \ya^k} - \frac{1}{n}\vecr\tp\aalpha\PC\YY^k}^2
           + \frac{\beta^2}{n^2} \expec{\nt{\vecr\tp \EX^k}^2 | \calF_k }  \\
        \leq& \nt{\xa^k - \altil \gg^k - \xx^* + \altil\pr{\gg^k - \ya^k} - \frac{1}{n}\vecr\tp\aalpha\PC\YY^k}^2 
          + \frac{\cerr{2}\nt{\vecr}^2\beta^2}{n^2} \mt{\XX^k - \UU^k}^2  \\
        \leq& \frac{1}{1 - \altil \mu} \nt{\xa^k - \altil \gg^k - \xx^*}^2 + \frac{2\altil }{\mu} \nt{\gg^k - \ya^k}^2 
         + \frac{2}{\mu \altil n^2} \nt{\vecr\tp\aalpha\PC\YY^k}^2
        + \frac{\cerr{2}\nt{\vecr}^2\beta^2}{n^2} \mt{\XX^k - \UU^k}^2  \\
        \leq& \pr{1 - \altil\mu} \nt{\xa^k - \xx^*}^2 + \frac{2L^2\altil}{\mu n} \mR{\PR\XX^k  }^2  
          + \frac{2\nt{\vecr}^2\almax^2}{\mu n^2 \altil} \mt{\PC\YY^k}^2
         + \frac{\cerr{2}\nt{\vecr}^2\beta^2}{n^2} \mt{\XX^k - \UU^k}^2   \\
        \leq& \pr{1 - c_1 \almax} \nt{\xa^k - \xx^*}^2 + c_2 \almax \mR{\PR\XX^k}^2   
          + c_3\almax\mC{\YY^k - \vecc\ya^k}^2 + c_4\beta^2 \mt{\XX^k - \UU^k}^2  ,
}
where
\seql{\label{eq: def c 1}}{
  c_1 = \frac{\mu}{w},\ c_2 = \frac{2\pr{\vecr\tp\vecc}L^2}{\mu n^2},\ c_3 =\frac{2\nt{\vecr}^2w}{\mu n^2},\ c_4 = \frac{\cerr{2}\nt{\vecr}^2}{n^2}.
}
Here the first equality uses Lemma~\ref{lem: condi expec} to eliminate the cross terms;
the first inequality is by (\ref{eq: bound for Ek});
the second inequality is by $\mt{\AA + \BB + \CC}^2 \leq \theta\mt{\AA}^2 + \frac{2\theta}{\theta - 1}\pr{\mt{\BB}^2 + \mt{\CC}^2}$ which follows by setting $\theta = \frac{1}{1 - \altil\mu}$ in Lemma~\ref{lem: Hilbert mm AM GM};
the third inequality is by (\ref{eq: gradient norm small}), (\ref{eq: dd gg in xa expansion});
and we used the relation $\altil \geq \almax / w$ and the fact $\mt{\PC\YY^k} \leq \mC{\PC\YY^k}$ from Lemma~\ref{def nR nC} in the last inequality.

Bounding $\expec{\mR{\PR\XX^{k+1}}^2}$: \
taking conditional expectation on both sides of (\ref{eq: x - one xa expansion}),
\seql{\label{eq: x consensus bound}}{
        &\expec{\mR{\PR\XX^{k+1}}^2 | \calF_k} \\
        =&\mR{\PR\Rb \PR\XX^k  - \PR\aalpha \YY^k}^2
           + \beta^2 \expec{\mR{\PR\RR\EX^k}^2 | \calF_k}  \\
        \leq& \frac{\nR{\PR\Rb}^2}{1 - \theR\beta}\mR{\PR\XX^k}^2
          + \frac{\nR{\PR}^2}{\theR\beta}\mR{\aalpha\YY^k}^2  
          + \beta^2 \nR{\PR\RR}^2\expec{\mR{\EX^k}^2 | \calF_k}  \\
        \leq& \pr{1 - \theR \beta} \mR{\PR\XX^k}^2 + \frac{c_5\almax^2}{\beta} \mt{\YY^k}^2  
          + c_6\beta^2 \mt{\UU^k - \XX^k}^2,
}
where
\seql{\label{eq: def c 2}}{
    c_5 = \frac{\traRt^2}{\theR},\  c_6 = \cerr{2}\traRt^2 .
}
The first equality uses Lemma~\ref{lem: condi expec} to eliminate cross terms;
the first inequality is by setting $\theta = \frac{1}{1 - \theR\beta}$ in Lemma~\ref{lem: Hilbert mm AM GM};
the last inequality is by Lemma~\ref{def nR nC} and (\ref{eq: bound for Ek}).

Bounding $\expec{\mC{\PC\YY^{k+1}}^2}$: \
using (\ref{eq: X update k+1}), we have
\seq{
  \XX^{k+1} - \XX^k + \beta\RR\WW^k  = \beta\pr{\RR - \II} \XX^k - \aalpha \YY^k
}
is measurable with respect to $\calF_k$.
Now, we have
\seql{\label{eq: XXk+1 XXcalFkmb}}{
     &\mt{\XX^{k+1} - \XX^k + \beta\RR\WW^k }^2  
     = \mt{\beta\pr{\RR - \II} \XX^k - \aalpha \YY^k }^2 
     = \mt{\beta\pr{\RR - \II} \pr{\XX^k - \one\xa^k} - \aalpha \YY^k }^2  \\
     \leq& 2\beta^2\nm{2}{\RR - \II}^2 \mt{\PR\XX^k  }^2 + 2\almax^2\mt{\YY^k }^2  
     \leq 8n\beta^2\mR{\PR\XX^k}^2 + 2\almax^2\mt{\YY^k}^2  ,
}
where
we used the fact $\mt{\cdot} \leq \mR{\cdot}$ from Lemma~\ref{def nR nC} and the fact $\nt{\RR - \II}^2 \leq n\ni{\RR - \II}^2 \leq 4n$ in the last inequality.
Thus, by Lemma~\ref{lem: condi expec} and (\ref{eq: bound for Ek}),
\seql{\label{eq:Xk+1-Xk}}{
     \expec{\mt{\XX^{k+1} - \XX^k }^2|\calF_k}  
     =& \mt{\XX^{k+1} - \XX^k + \beta\RR\EX^k }^2 + \beta^2\expec{\mt{\RR\WW^k }^2|\calF_k}  \\
     \leq&  \mt{\XX^{k+1} - \XX^k + \beta\RR\WW^k }^2 + \cerr{2} n \beta^2\mt{\UU^k - \XX^k}^2  ,
}
where we also used the fact $\nt{\RR}^2 \leq n\ni{\RR}^2 = n $ in the last inequality.

Taking conditional expectation on both sides of (\ref{eq: y - one ya expansion}) yields
\seq{
        &\expec{\mC{\PC\YY^{k+1}}^2 | \calF^+_k } \\
        =&\mC{\PC\Cg \PC\YY^k + \PC\pr{\na \FF\pr{\XX^{k+1}} - \na \FF\pr{\XX^k}}}^2  
         + \gamma^2 \expec{\mC{\pr{\II - \CC} \EE^k }^2 | \calF^+_k }  \\
        \leq& \frac{\nC{\PC\Cg}^2}{1 - \theC\gamma} \mC{ \PC\YY^k }^2  
         + \frac{\nC{\PC}^2}{\theC\gamma}  \mC{ \na \FF\pr{\XX^{k+1}} - \na \FF\pr{\XX^k}}^2 
          + \gamma^2\nC{\II - \CC}^2 \expec{ \mC{\EE^k}^2 | \calF^+_k }  \\
        \leq& \pr{1 - \theC\gamma} \mC{\PC\YY^k}^2 + \frac{c_7}{\gamma} \mt{\XX^{k+1} - \XX^k}^2  
                    + c_8\gamma^2 \mt{\YY^k}^2 ,
}
where
\seql{\label{eq:defc3}}{
  c_7 = \frac{L^2\traCt^2}{\theC},\ c_8 = 9\cerr{2}\traCt^2  .
}
Here the first equality uses Lemma~\ref{lem: condi expec} to eliminate cross terms;
the first inequality is by  Lemma~\ref{lem: Hilbert mm AM GM};
the last inequality above is from (\ref{eq: bound for Ek}), Lemma~\ref{def nR nC} and the $L$-smoothness in Assumption~\ref{assp: require each fi}.

Then, by substituting (\ref{eq: XXk+1 XXcalFkmb}), (\ref{eq:Xk+1-Xk}),
\seql{\label{eq: Yk+1 consensus Fk bound}}{
        &\expec{\mC{\PC\YY^{k+1}}^2 | \calF_k } \\
        \leq& \pr{1 - \theC\gamma} \mC{\PC\YY^k}^2 + \frac{c_7}{\gamma} \expec{\mt{\XX^{k+1} - \XX^k}^2|\calF_k} 
         + c_8\gamma^2 \mt{\YY^k}^2 \\
        \leq& \pr{1 - \theC\gamma} \mC{\PC\YY^k}^2 + \pr{c_8\gamma^2 + \frac{2c_7\almax^2}{\gamma}} \mt{\YY^k}^2  
         + \frac{8n c_7 \beta^2}{\gamma}\mR{\PR\XX^k}^2 + \frac{\cerr{2}n c_7 \beta^2}{\gamma}\mt{\UU^k - \XX^k}^2  .
}

Bounding $\expec{\mt{\XX^k - \UU^k}^2}$: \
%
by (\ref{eq: U update k+1}) and taking conditional expectation, we have
\seql{\label{eq: UU bound}}{
     &\expec{\mt{\XX^{k+1} - \UU^{k+1}}^2 | \calF_k } \\
     =& \mathbb{E}\Big[\mmB(1 - \eta)\pr{\XX^k - \UU^k} + \XX^{k+1} - \XX^k  
      + \beta\RR\WW^k + \pr{\eta\II - \beta\RR} \EX^k\mmB_2^2 | \calF_k \Big] \\
     \leq& \mt{(1 - \eta)\pr{\XX^k - \UU^k} + \XX^{k+1} - \XX^k + \beta\RR\WW^k}^2 
              + \nt{\eta\II - \beta\RR}^2 \expec{\mt{\EX^k}^2 | \calF_k }  \\
    \leq& \mt{(1 - \eta)\pr{\XX^k - \UU^k} + \XX^{k+1} - \XX^k + \beta\RR\WW^k}^2  
     + 2\cerr{2}\pr{\eta^2 + n\beta^2}\mt{\XX^k - \UU^k}^2  \\
    \leq& \pr{1 - \eta}\mt{\XX^k - \UU^k}^2 + \frac{1}{\eta}\mt{\XX^{k+1} - \XX^k + \beta\RR\WW^k}^2 
     + 2\cerr{2}\pr{\eta^2 + n\beta^2} \mt{\XX^k - \UU^k}^2  \\
    \leq& \pr{1 - \eta + 2\cerr{2}\eta^2 + 2n\cerr{2}\beta^2 } \mt{\XX^k - \UU^k}^2  
     + \frac{8n\beta^2}{\eta}\mR{\PR\XX^k}^2 + \frac{2\almax^2}{\eta}\mt{\YY^k}^2 ,
}
where
the first inequality is by Lemma~\ref{lem: condi expec};
the second inequality is by (\ref{eq: bound for Ek}) and the fact $\nt{\RR}^2 \leq n\ni{\RR}^2 = n $;
the third inequality uses Lemma~\ref{lem: Hilbert mm AM GM};
the last inequality is by substituting (\ref{eq: XXk+1 XXcalFkmb}).

Now, we bound the term $\mt{\YY^k}$ which occurs frequently above.
By Assumption~\ref{assp: require each fi} and the fact $\one\tp\na\FF\pr{\XX^*} = \zero\tp$, we have
\seql{\label{eq: gg in xa expansion}}{
       \nt{\gg^k} &= \frac{1}{n}\nt{\one\tp\na\FF\pr{\one\xa^k} - \one\tp\nabla\FF\pr{\XX^*}} 
       \leq \frac{L}{n} n \nt{\xa^k - \xx^*} = L\nt{\xa^k - \xx^*}.
}
By decomposing $\YY^k = \pr{\YY^k - \vecc\ya^k} + \vecc\pr{\ya^k - \gg^k} + \vecc\gg^k $ and using (\ref{eq: dd gg in xa expansion})(\ref{eq: gg in xa expansion}), we have
\seql{\label{eq: mt YYk}}{
     \mt{\YY^k}^2  
    \leq& 3\mt{\PC\YY^k}^2 + 3\nt{\vecc}^2\nt{\ya^k - \gg^k}^2 + 3\nt{\vecc}^2\nt{\gg^k}^2  \\
    \leq& 3\mC{\PC\YY^k}^2 + c_9\mR{\PR\XX^k}^2 + c_{10}\nt{\xa^k - \xx^*}^2,
}
where
\seql{\label{eq:defc4}}{
  c_9 = \frac{3L^2\nt{\vecc}^2}{n},\ c_{10} = 3L^2\nt{\vecc}^2  .
}

The relation (\ref{eq:cYmat1}) is by taking expectation on both sides of (\ref{eq: mt YYk}).
Combining (\ref{eq: xa square expec 1}), (\ref{eq: x consensus bound}), (\ref{eq: Yk+1 consensus Fk bound}) and (\ref{eq: UU bound}) and taking full expectation yields (\ref{eq:cmat2}). \qed

\subsection{Proof of Lemma~\ref{lem:bmatineq1}}\label{sec:pfbmatineq1}
Define
\begin{align*}
   & \AW = \expec{\AW^k | \calD_k},
    \Vara = \expec{\nt{\II - \RW^k + \RU^k - \RR}^2|\calD_k}, 
    \Varb = \expec{\nt{\RU^k - \RP^k}^2|\calD_k},
    \Varc = \expec{\nt{\RP^k}^2|\calD_k} , \\
   & \Vard = \expec{\nt{\AW - \AW^k}^2 | \calD_k},
    \Vare = \expec{\nt{\CQ^k - \IW^k}^2|\calD_k}, 
     \Varf = \expec{\nt{\II - \CC + \CQ^k - \IW^k}^2|\calD_k}  .
\end{align*}
Note that $\AW$, $\Vara$-$\Varf$ are constants independent of $k$.

It follows by Lemma~\ref{assp:bindept.} that
\seql{\label{eq:bmexzero}}{
  \begin{split}
      & \expec{\II - \RW^k + \RU^k - \RR |\calD_k} = 0,\
       \expec{\RU^k - \RP^k|\calD_k} = 0 .
  \end{split}
}

Since $\one\tp\pr{\CQ^k - \IW^k} = \zero\tp$, by induction, there holds
$
   \ya^k = \frac{1}{n}\one\tp\gF{k} .
$
Now, we can expand (\ref{eq:bXupdate}) as
\seql{\label{eq:bXexpand1}}{
     \XX^{k+1} =& \Rb\XX^k - \alp\AW\YY^k + \alp\pr{\AW - \AW^k}\YY^k  
       + \beta\pr{\II - \RW^k + \RU^k - \RR}\XX^k
       + \beta\RP^k\pr{\UU^k + \PP^k - \XX^k } \\
       & + \beta\pr{\RU^k - \RP^k}\pr{\UU^k - \XX^k}
        + \beta\pr{\RW^k\UR^k - \RU^k\UU^k}  \\
      =& \Rb\XX^k - \alp\AW\YY^k + \ro^k  ,
}
where
\seql{\label{defro1}}{
  \ro^k =& \alp\pr{\AW - \AW^k}\YY^k
       + \beta\pr{\II - \RW^k + \RU^k - \RR}\PR\XX^k  
        + \beta\pr{\RU^k - \RP^k}\pr{\UU^k - \XX^k} - \beta\RP^k\EX^k ,
}
and in the second equality, we used the fact $\RW^k\UR^k = \RW^k\RR\UU^k = \RU^k\UU^k$ and $\pr{\II - \RW^k + \RU^k - \RR}\one = \zero$.

By multiplying $\frac{\vecr\tp}{n}$ on both sides of (\ref{eq:bXexpand1}) and defining $\alb = \frac{\alp\vecr\tp\AW\vecc}{n}$,
\seq{
     &\xa^{k+1} = \xa^k - \frac{\alp}{n}\vecr\tp\AW\YY^k + \frac{1}{n}\vecr\tp\ro^k  
       = \xa^k - \alb\gg^k + \alb\pr{\gg^k - \ya^k}
         - \frac{\alp}{n}\vecr\tp\AW\PC\YY^k + \frac{1}{n}\vecr\tp\ro^k  .
}
By Lemma 10 in~\cite{qu2017harnessing}, for $\alb \leq \frac{2}{\mu + L}$,
\seql{\label{eq: gradient norm small broadcast}}{
  \nt{\xa^k - \alb \gg^k - \xx^*} \leq \pr{1 - \alb \mu} \nt{\xa^k - \xx^*}.
}
Using (\ref{eq:bmexzero}) with the fact $\expec{\EX^k | \calD_k^+} = \zero$ from Lemma~\ref{assp:bindept.}, we have
\seql{\label{eq:rocalDk0}}{
  \expec{\ro^k | \calD_k} = 0 .
}
By taking conditional expectation and using (\ref{eq:rocalDk0}) to eliminate cross-terms,
\seq{
    &\expec{\nt{\xa^{k+1} - \xx^*}^2 | \calD_k}
    =  \expec{\nt{\frac{1}{n}\vecr\tp\ro^k}^2 | \calD_k} 
    + \nt{\xa^k - \alb\gg^k - \xx^* + \alb\pr{\gg^k - \ya^k} - \frac{\alp}{n}\vecr\tp\AW\PC\YY^k}^2 .
}

Similarly to (\ref{eq: xa square expec 1}), there holds
\seql{\label{eq:bxaup}}{
     &\expec{\nt{\xa^{k+1} - \xx^*}^2 | \calD_k}  \\
     \leq&\nt{\xa^k - \alb\gg^k - \xx^* + \alb\pr{\gg^k - \ya^k} - \alp\vecr\tp\AW\PC\YY^k}^2 
       + \frac{\nt{\vecr}^2}{n^2}\expec{\nt{\ro^k}^2 | \calD_k}  \\
     \leq& \frac{1}{1 - \alb\mu}\nt{\xa^k - \alb\gg^k - \xx^*}^2 + \frac{2\alb}{\mu}\nt{\gg^k - \ya^k}^2 
      + \frac{2\alb}{\mu}\nt{\vecr\tp\AW\PC\YY^k}^2
       + \frac{\nt{\vecr}^2}{n^2}\expec{\nt{\ro^k}^2 | \calD_k} \\
     \leq& \pr{1 - \alb\mu}\nt{\xa^k - \xx^*}^2 + \frac{2\alb L^2}{\mu n}\mR{\PR\XX^k}^2  
      + \frac{2\alb}{\mu}\nt{\vecr}^2\nt{\AW}^2\mt{\PC\YY^k}^2
       + \frac{\nt{\vecr}^2}{n^2}\expec{\nt{\ro^k}^2 | \calD_k}  \\
     \leq& \pr{1 - d_1\alp}\nt{\xa^k - \xx^*}^2 + d_2\alp\mR{\PR\XX^k}^2  
      + d_3\alp\mC{\PC\YY^k}^2 + d_4\expec{\nt{\ro^k}^2 | \calD_k} ,
}where
\seql{\label{defd1}}{
  & d_1 = \frac{\vecr\tp\AW\vecc \mu}{n},\ d_2 = \frac{2\vecr\tp\AW\vecc L^2}{\mu n^2},  
    d_3 = \frac{2\vecr\tp\AW\vecc \nt{\vecr}^2}{\mu n},
   d_4 = \frac{\nt{\vecr}^2}{n^2}  .
}
The second inequality above is by Lemma~\ref{lem: Hilbert mm AM GM};
the third inequality used (\ref{eq: dd gg in xa expansion}) and (\ref{eq: gradient norm small broadcast});
the last inequality above used the relation $\nt{\AW} < 1$ which follows from the fact that $\AW$ is a diagonal matrix with all diagonal entries within $(0, 1)$.

To bound $\expec{\mR{\PR\XX^{k+1}}|\calD_k}$, we first multiply $\PR$ on both sides of (\ref{eq:bXexpand1}), which yields
\seq{
     &\PR\XX^{k+1} =
     \PR\Rb \PR\XX^k - \alp\PR\AW\YY^k + \PR\ro^k  .
}
Using (\ref{eq:rocalDk0}) to eliminate cross-terms, we obtain
\seq{
     \expec{\mR{\PR\XX^{k+1}}^2|\calD_k} =&
     \mR{\PR\Rb \PR\XX^k - \alp\PR\AW\YY^k}^2  
      + \expec{\mR{\PR\ro^k}^2 | \calD_k}  .
}
Thus, we have
\seql{\label{eq:bXcon1}}{
   & \expec{\mR{\PR\XX^{k+1}}^2|\calD_k}  \\
   \leq&\mR{\PR\Rb \PR\XX^k - \alp\PR\AW\YY^k}^2 + \nR{\PR}^2\expec{\mR{\ro^k}^2 | \calD_k}  \\
   \leq& \frac{\nR{\PR\Rb}^2}{1 - \beta\theR}\mR{\PR\XX^k}^2 + \frac{\alp^2\nR{\PR}^2\nR{\AW}^2}{\beta\theR}\mR{\YY^k}^2 
    + \nR{\PR}^2\expec{\mR{\ro^k}^2 | \calD_k}  \\
   \leq& \pr{1 - \theR\beta}\mR{\PR\XX^k  }^2 + \frac{\alp^2}{\beta\theR}\mR{\YY^k}^2 
    + \traRt^2\expec{\mt{\ro^k}^2 | \calD_k} ,
}
where the second inequality is by Lemma~\ref{lem: Hilbert mm AM GM};
the last inequality uses Lemma~\ref{def nR nC} and Lemma~\ref{lem: Hilbert mm AM GM}.

Next, we proceed to bound $\expec{\mt{\XX^{k+1} - \XX^k}^2|\calD_k}$.
By (\ref{eq:bXexpand1}) and the fact $\pr{\RR - \II}\one = \zero$,
\seq{
     \XX^{k+1} - \XX^k &= \beta\pr{\RR - \II}\XX^k - \alp\AW\YY^k + \ro^k  
      = \beta\pr{\RR - \II}\PR\XX^k - \alp\AW\YY^k + \ro^k.
}
Again, by taking conditional expectation and using (\ref{eq:rocalDk0}) to eliminate the cross-terms,
\seql{\label{eq:bXk+1kdiff}}{
     &\expec{\mt{\XX^{k+1} - \XX^k}^2 | \calD_k} \\
     =& \mt{\beta\pr{\RR - \II}\PR\XX^k - \alp\AW\YY^k}^2 + \expec{\mt{\ro^k}^2|\calD_k}   \\
     \leq& 2\beta^2\nR{\RR - \II}^2\mR{\PR\XX^k}^2 + 2\alp^2\nt{\AW}^2\mt{\YY^k}^2 
      + \expec{\mt{\ro^k}^2|\calD_k}  \\
     \leq& 18\beta^2\mR{\PR\XX^k}^2 + 2\alp^2\mt{\YY^k}^2
     + \expec{\mt{\ro^k}^2|\calD_k}  ,
}
where we used the fact $\nR{\RR - \II} \leq 3  $ from Lemma~\ref{def nR nC} in the last inequality.
\newcommand\ctc{\calW}
To bound $\expec{\mC{\PC\YY^{k+1}}^2|\calD_k}$, we first expand (\ref{eq:bYupdate}) using $\QQ^k = \YY^k - \EY^k$, then
\seq{
       \YY^{k+1} =& \YY^k + \gamma\pr{\CQ^k - \IW^k}\pr{\YY^k - \EY^k}  + \gF{k+1} - \gF{k}   \\
       =& \ctc^k - \gamma\pr{\CQ^k - \IW^k}\EY^k
         + \gF{k+1} - \gF{k},
}
where $\ctc^k = \Cg\YY^k + \gamma\pr{\II - \CC + \CQ^k - \IW^k}\YY^k. $

By multiplying $\PC$, we have
\seq{
        &\PC\YY^{k+1} = \PC\ctc^k + \PC\pr{\gF{k+1} - \gF{k}} 
         - \gamma\PC\pr{\CQ^k - \IW^k}\EY^k  .
}
Then, taking conditional expectation on both sides of the above equation and using (\ref{eq:bEY}) yield
\begin{equation*}
  \small
  \begin{split}
     &\expec{\mC{\PC\YY^{k+1}}^2|\calD_k^{++}} \\
     =& \mC{\PC\ctc^k + \PC\pr{\gF{k+1} - \gF{k}}}^2  
      + \gamma^2\expec{\mC{\PC\pr{\CQ^k - \IW^k}\EY^k}^2|\calD_k^{++}} \\
     \leq& \mC{\PC\ctc^k + \PC\pr{\gF{k+1} - \gF{k}}}^2  
      + \frac{\cerr{2}\Vare\traCt^2\gamma^2}{n}\mt{\YY^k}^2 ,
  \end{split}
\end{equation*}
where we also used $\nC{\PC} = 1$ from Lemma~\ref{def nR nC} in the last inequality.

It follows by the definitions that
$
  \expec{\II - \CC + \CQ^k - \IW^k | \calD_k} = \zero .
$
Thus,
\seq{
       \expec{\mC{\PC\ctc^k}^2 | \calD_k}
      =& \mC{\PC\Cg \PC\YY^k}^2  
       + \gamma^2\expec{\mC{\PC\pr{\II - \CC + \CQ^k - \IW^k}\YY^k}^2 | \calD_k}  \\
      \leq& \pr{1 - \theC\gamma}^2\mC{\PC\YY^k}^2 + \gamma^2\Varf\traCt^2\mt{\YY^k}^2  .
}

By combining the above equations and using the $L$-smoothness, we obtain
\begin{equation}\label{eq:bYcon1}
    \small
    \begin{split}
        &\expec{\mC{\PC\YY^{k+1}}^2|\calD_k} \\  
      \leq& \frac{1}{1 - \theC\gamma}\mathbb{E}\Big[\mmB\PC\Cg \PC\YY^k  
       + \gamma\PC\pr{\II - \CC + \CQ^k - \IW^k}\YY^k\mmB_{\rmC}^2 | \calD_k\Big]  \\ 
       & + \frac{\traCt^2}{\theC\gamma}\expec{\mt{\gF{k+1} - \gF{k}}^2|\calD_k}  
       + \frac{\cerr{2}\Vare\traCt^2\gamma^2}{n}\mt{\YY^k}^2  \\
      \leq& \pr{1 - \theC\gamma}\mC{\PC\YY^k}^2 + \frac{d_5\gamma^2}{1 - \theC\gamma}\mt{\YY^k}^2  
        + \frac{d_6}{\gamma}\expec{\mt{\XX^{k+1} - \XX^k}^2|\calD_k} + d_7\gamma^2 \mt{\YY^k}^2  , 
    \end{split}
\end{equation}
where
\seql{\label{defd2}}{
  d_5 = \Varf\traCt^2,\ d_6 = \frac{\traCt^2L^2}{\theC},\ d_7 = \frac{\cerr{2}\Vare\traCt^2}{n}  .
}

To bound $\expec{\mt{\UU^{k+1} - \XX^{k+2}}^2|\calD_k}$, we expand (\ref{eq:bUupdate}) using the relation $\EX^k = \XX^k - \PP^k - \UU^k$, i.e.,
\seq{
  \UU^{k+1} = \pr{1 - \eta}\UU^k + \eta\pr{\II - \IW^k}\UU^k + \eta\IW^k\XX^k - \eta\IW^k\EX^k .
}
Then,
\seq{
        &\UU^{k+1} - \XX^{k+1} \\
         =& \pr{1 - \eta}\pr{\UU^k - \XX^k}  
         + \eta\pr{\II - \IW^k}\pr{\UU^k - \XX^k} - \pr{\XX^{k+1} - \XX^k} - \eta\IW^k\EX^k  \\
        =& \pr{1 - \eta}\pr{\UU^k - \XX^k} + \eta\pr{\II - \IW^k}\pr{\UU^k - \XX^k} 
         - \pr{\XX^{k+1} - \XX^k + \beta\RP^k\EX^k} - \pr{\eta\IW^k - \beta\RP^k}\EX^k  .
}
By the definition of $\ro^k$ in (\ref{defro1}), we have
\seq{
    &\ro^k + \beta\RP^k\EX^k  
    = \alp\pr{\AW - \AW^k}\YY^k + \beta\pr{\RU^k - \RP^k}\pr{\UU^k - \XX^k}  
     + \beta\pr{\II - \RW^k + \RU^k - \RR}\PR\XX^k
}
is measurable with respect to $\calD_k^+$.
Then,
\seq{
    \XX^{k+1} - \XX^k + \beta\RP^k\EX^k &= \beta\pr{\RR - \II}\PR\XX^k  
     - \alp\AW\YY^k + \ro^k + \beta\RP^k\EX^k
}
is measurable with respect to $\calD_k^+$.
Define $\uxa^k = \pr{1 - \eta}\pr{\UU^k - \XX^k} + \eta\pr{\II - \IW^k}\pr{\UU^k - \XX^k}. $
Now, by taking conditional expectation and using (\ref{eq:bEX}),
\seq{
      & \expec{\mt{\UU^{k+1} - \XX^{k+1}}^2|\calD_k^+} \\
      =& \mmB\uxa^k
        - \pr{\XX^{k+1} - \XX^k + \beta\RP^k\EX^k}\mmB^2 
       + \expec{\mt{\pr{\eta\IW^k - \beta\RP^k}\EX^k}^2|\calD_k^+}  \\
      \leq& \mmB\uxa^k
       - \pr{\XX^{k+1} - \XX^k + \beta\RP^k\EX^k}\mmB^2 
        + 2\cerr{2}\pr{\eta^2n + \frac{\beta^2}{n}\nt{\RP^k}^2 }\mt{\UU^k - \XX^k}^2  ,
}
where we also used the fact $\nt{\IW^k} = n$ in the last inequality.

Since $\expec{\II - \IW^k|\calD_k} = \zero$ and $\nt{\IW^k - \II} = n - 1$,
\seq{
      &\expec{\mt{\uxa^k}^2|\calD_k}  \\
      =& \pr{1 - \eta}^2\mt{\UU^k - \XX^k}^2  
         + \eta^2\expec{\mt{\pr{\II - \IW^k}\pr{\UU^k - \XX^k}}^2|\calD_k}   \\
      \leq& \pr{1 - \eta}^2\mt{\UU^k - \XX^k}^2   
         + \eta^2\expec{\nt{\II - \IW^k}^2|\calD_k}\mt{\UU^k - \XX^k}^2  \\
      =& \pr{1 - \eta}^2\mt{\UU^k - \XX^k}^2 + \eta^2\pr{n-1}^2\mt{\UU^k - \XX^k}^2  .
}
And
\seq{
     &\expec{\mt{\XX^{k+1} - \XX^k + \beta\RP^k\EX^k}^2|\calD_k} \\
     \leq& 2\expec{\mt{\XX^{k+1} - \XX^k}^2|\calD_k} + 2\beta^2\expec{\mt{\RP^k\EX^k}|\calD_k}     \\
     \leq& 2\expec{\mt{\XX^{k+1} - \XX^k}^2|\calD_k} + \frac{2\cerr{2}\beta^2\Varc}{n}\mt{\UU^k - \XX^k}^2  .
}

Combining the above equations, we have
\begin{equation}\label{eq:bUXcon1}
  \small
  \begin{split}
      & \expec{\mt{\UU^{k+1} - \XX^{k+1}}^2|\calD_k}  \\
      \leq& \frac{1}{1 - \eta}\expec{\mt{\uxa^k}^2|\calD_k}
        + \frac{1}{\eta}\expec{\mt{\XX^{k+1} - \XX^k + \beta\RP^k\EX^k}^2|\calD_k}  
       + 2\cerr{2}\pr{\eta^2n + \frac{\beta^2\Varc}{n} }\mt{\UU^k - \XX^k}^2  \\
      \leq& \pr{1 - \eta + \frac{\eta^2(n-1)^2}{1 - \eta} + d_8\frac{\beta^2}{\eta} + 2\cerr{2}n\eta^2 + d_8\beta^2}\mt{\UU^k - \XX^k}^2  
         + \frac{2}{\eta}\expec{\mt{\XX^{k+1} - \XX^k}^2|\calD_k}  ,
  \end{split}
\end{equation}
where
\seql{\label{eq:defd3}}{
  d_8 = \frac{2\cerr{2}\Varc}{n}.
}

For the variable $\ro^k$ defined in (\ref{defro1}), we have
\begin{equation}\label{eq:bxat2}
    \small
  \begin{split}
      & \expec{\mt{\ro^k}^2 | \calD_k} \\
      =& \mathbb{E}\Big[\mmB\alp\pr{\AW - \AW^k}\YY^k + \beta\pr{\II - \RW^k + \RU^k - \RR}\PR\XX^k  
          + \beta\pr{\RU^k - \RP^k}\pr{\UU^k - \XX^k} \mmB^2 | \calD_k\Big] \\  
       & + \beta^2\expec{\mt{\RP^k\EX^k}^2|\calD_k}  \\
      \leq& 3\alp^2\Vard\mt{\YY^k}^2 + 3\beta^2\expec{\nt{\II - \RW^k + \RU^k - \RR}^2|\calD_k}\mR{\PR\XX^k}^2  
        + 3\beta^2\expec{\nt{\RU^k - \RP^k}^2|\calD_k}\mt{\UU^k - \XX^k}^2  \\
         & + \beta^2\expec{\nt{\RP^k}^2 \expec{\mt{\EX^k}^2|\calD_{k}^+} | \calD_k}   \\
      \leq& 3\alp^2\Vard\mt{\YY^k}^2 + 3\beta^2\Vara\mR{\PR\XX^k}^2 + d_{9}\beta^2\mt{\UU^k - \XX^k}^2  ,
  \end{split}
\end{equation}
where
\seql{\label{eq:defd4}}{
  d_{9} = 3\Varb +\frac{\cerr{2}\Varc}{n} .
}
Here we used the fact $\expec{\EX^k | \calD_k^+} = \zero$ from Lemma~\ref{assp:bindept.} in the first equality and (\ref{eq:bEX}) in the last inequality.

The relation (\ref{eq:bYkAb1}) is by taking full expectation on both sides of (\ref{eq: mt YYk});
the relation (\ref{eq:brokAb1}) is derived by taking full expectation on both sides of (\ref{eq:bxat2});
the relation (\ref{eq:bXk+1kAb1}) is by taking full expectation on both sides of (\ref{eq:bXk+1kdiff});
the relation (\ref{eq:bdk+1Ab1}) can be derived by combining (\ref{eq:bxaup}), (\ref{eq:bXcon1}), (\ref{eq:bYcon1}) and (\ref{eq:bUXcon1}) and taking full expectation.\qed

\end{document}